\documentclass[reqno,oneside]{amsart}

\usepackage{amssymb,graphicx,braket,amsmath,amsfonts,enumitem}
\usepackage[mathscr]{euscript}
\usepackage{relsize}
\usepackage[hidelinks]{hyperref}
\usepackage[utf8]{inputenc}
\usepackage{braket}
\usepackage{color}
\usepackage{amsaddr}
\usepackage{tikz}
\usetikzlibrary{shapes,arrows,er,positioning}
\usetikzlibrary{shapes.multipart}
\usepackage{stmaryrd}
\usepackage{fancyhdr}
\usepackage{array}
\usepackage{wasysym}
\usepackage{bm}
\usepackage[all]{xy}
\usepackage[labelfont={rm},labelformat=simple]{subcaption}

\newcommand*{\xdash}[1][3em]{\rule[0.5ex]{#1}{0.55pt}}
\newcommand{\zbar}{%
  \text{\ooalign{\raisebox{0.19ex}{\kern0.22em \xdash[0.3em]}\cr$Z$\cr}}%
}

\usepackage[lmargin=1.1in,rmargin=1.1in,tmargin=1.35in,bmargin=1.35in]{geometry}

\usetikzlibrary{arrows}

\theoremstyle{plain}
\newtheorem{theorem}{Theorem}[section]
\theoremstyle{plain}
\newtheorem{proposition}[theorem]{Proposition}
\newtheorem*{proposition*}{Proposition}
\newtheorem*{theorem*}{Theorem}
\newtheorem{lemma}[theorem]{Lemma}

\newtheorem*{assumption*}{Assumption}

\newtheorem{conjecture}[theorem]{Conjecture}

\theoremstyle{definition} 
\newtheorem{definition}[theorem]{Definition}

\newtheorem{remark}[theorem]{Remark}

\numberwithin{equation}{section}

\makeatletter
\renewcommand{\section}{\@startsection
{section}
{1}
{\z@}
{-\baselineskip}
{0.8\baselineskip}
{\centering\scshape\large}} 

\renewcommand{\subsection}{\@startsection
{subsection}
{2}
{\z@}
{-0.8\baselineskip}
{0.5\baselineskip}
{\normalfont \bf \normalsize}} 

\renewcommand{\subsubsection}{\@startsection
{subsubsection}
{3}
{\z@}
{-0.8\baselineskip}
{0.5\baselineskip}
{\normalfont \bf \normalsize}} 
\makeatother

\usepackage[draft]{}

\begin{document}
\title[Refined BPS structures and topological recursion]{Refined BPS structures and topological recursion --- \\ the Weber and Whittaker curves}

\author{Omar Kidwai}
\address{\begin{center}School of Mathematics, Watson Building, University of Birmingham\end{center} Edgbaston, Birmingham, B15 2TT, United Kingdom}
\address{\vspace{-2mm}and\vspace{-2mm}}

\address{\begin{center}Theoretical Sciences Visiting Program (TSVP), Okinawa Institute of Science and Technology\end{center} 1919-1 Tancha, Onna-son, Kunigami-gun,
Okinawa, 904-0495, Japan}
\email{o.kidwai@bham.ac.uk}

\author{Kento Osuga}
\address{\begin{center}Graduate School of Mathematical Sciences,  University of Tokyo\end{center} 
3-8-1 Komaba, Meguro-ku, Tokyo, 153-8914, Japan}\email{osuga@ms.u-tokyo.ac.jp}

\begin{abstract}
We study properties of the recently established refined topological recursion for some simple spectral curves associated to quadratic differentials. We prove explicit formulas for the free energy and Voros coefficients of the corresponding quantum curves, and conjecture expressions for all other (smooth) genus zero degree two curves. The results can be written in terms of Bridgeland's notion of refined BPS structure associated to the same initial data, together with a quantum correction to the central charge. The corresponding ``invariants'' appear to be new, but their interpretation in terms of Donaldson-Thomas theory or QFT is not entirely clear.


\end{abstract}

\newpage\maketitle
\setcounter{tocdepth}{2}\tableofcontents

\section{Introduction}\label{sec:intro}

This paper generalizes a recently discovered relationship between the theory of topological recursion and BPS structures \cite{IWAKI2022108191} to the \emph{refined} setting. On one side, this involves considering the refined topological recursion recently established in \cite{kidwai2023quantum}, corresponding to a certain natural one-parameter deformation of the usual Chekhov-Eynard-Orantin topological recursion \cite{CEO,EO}, and on the other side roughly means we consider the structure appearing in refined Donaldson-Thomas theory \cite{kontsevich2008stability}. In the present paper we focus on the simplest examples of genus zero, the Weber and Whittaker curves\footnote{We also treat the Airy and degenerate Bessel curves, but they are almost trivial to handle.}, named for the corresponding classical ODEs appearing as their quantizations.

More precisely, our starting point is a meromorphic quadratic differential $\varphi$ on a compact Riemann surface $X$. From such an object, we may construct a corresponding \emph{refined spectral curve}, to which we may apply the refined topological recursion\footnote{Strictly speaking, in our previous work we only defined the refined recursion when the spectral curve is of degree two and genus zero. The higher genus case is expected to work (at least when $X=\mathbf{P}^1$), so long as the curve is (hyper)elliptic, though the higher degree case remains open.}. We will study the resulting invariants and the corresponding \emph{refined quantum curves} obtained in \cite{kidwai2023quantum}. The latter are Schr\"odinger-like operators containing a formal parameter $\hbar$ which annihilate a certain wavefunction constructed via the (refined) topological recursion. The goal of this paper is twofold: to determine these invariants explicitly, and to show that they admit an expression in terms of natural BPS-theoretic data constructed from the same initial data.

Our main examples in the present paper are the so-called Weber and Whittaker spectral curves, which arise as double covers associated to the quadratic differentials 
\begin{align}
 \varphi_{\rm Web}=\left(\frac{1}{4}x^2-m\right)dx^2&\hspace{-1cm} &\varphi_{\rm Whi}=\frac{x+4m}{4x}dx^2
\end{align}
respectively, where $x$ is a coordinate on $X=\mathbf{P}^1$, and $m\neq 0$ is a complex parameter. They are two of a class of fundamental examples, all of which have genus zero and degree two -- we call these the \emph{spectral curves of hypergeometric type}, since their quantizations are hypergeometric equations (and confluences thereof), depicted in Figure \ref{fig:confdiag}. The present paper treats the cases in which none of the poles are of second order (Whittaker, Weber, degenerate Bessel, and Airy).

\begin{figure}[h]
$$
\xymatrix@!C=25pt@R2pt{
&&
&& \underset{2(1)+(-6)}
{\text{Weber}} \ar@{~>}[rrd]
&&
\\
&& \color{lightgray}\underset{2(1)+(-2)+(-4)}
{\text{\color{lightgray} Kummer}} \ar@{->}@[lightgray][rru] \ar@{->}@[lightgray][rrd] \ar@{~>}@[lightgray][rrddd]
&&
&& \underset{(1)+(-5)}
{\text{Airy}}
\\
&&
&& \underset{(1)+(-1)+(-4)}
{\text{Whittaker}} \ar@{~>}[rru] \ar@{~>}[rrddd]
&&
&&
\\
\color{lightgray}\underset{2(1)+3(-2)}
{\text{\color{lightgray} Gauss}} \ar@{->}@[lightgray][rruu] \ar@{~>}@[lightgray][rrdd]
&&
&&
\\
&&
&& \color{lightgray}\underset{(1)+(-2)+(-3)}
{\text{\color{lightgray} Bessel}} 
\ar@{->}@[lightgray][rruuu]|(0.32)\hole \ar@{~>}@[lightgray][rrd]
&&
&&
\\
&& \color{lightgray}\underset{(1)+(-1)+2(-2)}
{\text{\color{lightgray} Degenerate Gauss}} 
\ar@{~>}@[lightgray][rru] \ar@{~>}@[lightgray][rrd] \ar@{~>}@[lightgray][rruuu]|(0.67)\hole
&&
&& \underset{(-1)+(-3)}
{\text{ Degenerate Bessel}}
\\
&&
&& \color{lightgray}\underset{2(-1)+(-2)}
{\text{\color{lightgray} Legendre}}\ar@{~>}@[lightgray][rru]
&&
\\
}
$$
\caption{Confluence diagram of hypergeometric type equations/spectral curves. The scheme $k(1)+\sum k_i(-r_i)$ below the name indicates the quadratic differential has $k$ simple zeroes and $k_i$ poles of order $r_i$. A straight line (resp. wavy line) in the figure
denotes the confluence of singular points
(resp. the coalescence of a branch point and a singular point). We treat the examples in black in the present paper, which all share the property that no second order poles are present.
}
\label{fig:confdiag}
\end{figure}
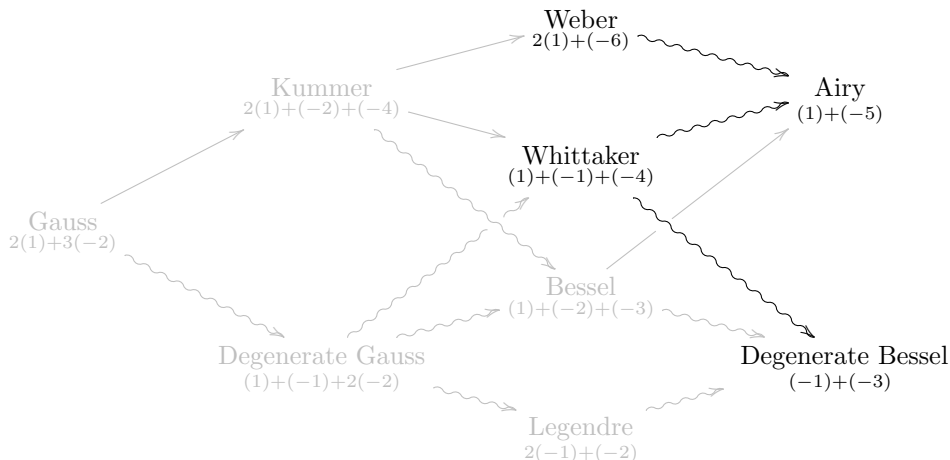

The results of \cite{iwaki2023voros,iwaki2019voros,IWAKI2022108191,iwaki2023topological} provided explicit formulas for invariants of the quantum curve, the so-called \emph{Voros coefficients}, as well as an explicit formula and relation to the \emph{free energy} of the topological recursion. We generalize these results for the mentioned curves, and conjecture them for the remaining ones. We find both Voros coefficients and the free energy of the refined topological recursion can be written in terms of a \emph{refined BPS structure}, in analogy with \cite{IWAKI2022108191}, together with an enhanced ``quantum central charge'' containing an $\hbar$-correction.

\subsection{Background}

\subsubsection{Refined BPS structures}



The notion of BPS structure was introduced by Bridgeland \cite{Bri19} axiomatizing the structure of Donaldson-Thomas (DT) invariants of a 3-Calabi-Yau triangulated category with a stability condition, and is a special case of a stability structure in the sense of Kontsevich-Soibelman \cite{kontsevich2008stability}. This same structure appeared in the work of Gaiotto-Moore-Neitzke \cite{gaiotto2010four,GAIOTTO2013239,Gaiotto:2012rg}, who studied the physics of a class of four-dimensional $\mathcal{N}=2$ supersymmetric QFTs, in which they describe the particle spectra of the theory at a Coulomb vacuum. Mathematical applications of their work abound, including a method for obtaining the hyperk\"ahler metric on the Hitchin moduli space, and a conjectural algorithm for determining Donaldson-Thomas invariants themselves. The refined version was less studied in physics but also appeared in a few works, in which the spin of the particles is studied \cite{Dimofte:2009tm,dimofte2010refined,galakhov2015spectral}.  There is a natural refined analogue of a BPS structure, which first appeared in \cite{BBS} and appears in refined Donaldson-Thomas theory. A (integral) refined BPS structure $(\Gamma,Z,\Omega)$ consists of
\begin{itemize}
\item 
a lattice $\Gamma$ equipped with an anti-symmetric pairing $\langle \cdot , \cdot \rangle$, 
\item 
a group homomorphism $Z : \Gamma \to {\mathbb C}$ called the \emph{central charge}, and
\item 
a map of sets $\Omega : \Gamma \to {\mathbb Z}[t,t^{-1}]$ called the \emph{(refined) BPS invariants}, 
\end{itemize}
satisfying certain conditions, crucially that $\Omega$ satisfies $\Omega(\gamma)=\Omega(-\gamma)$. Here $t$ is a formal variable.


In practice, we always consider BPS structures arising from quadratic differentials $\varphi$, via a curve $\Sigma$ called the \emph{spectral cover}. In this case $\Gamma$ is roughly $H_1({\Sigma},\mathbb{Z})$, $Z$ is the period of a certain one-form on $\Sigma$, and $\Omega$ is determined by the geometry of $\varphi$ (the behaviour of certain \emph{trajectories} on $X$). A definition of the latter was given in \cite{IWAKI2022108191}, generalizing Gaiotto-Moore-Neitzke \cite{GAIOTTO2013239} (see also \cite{bridgeland2015quadratic}). In this paper, we consider the relevant \emph{refined} BPS structure and propose a generalization of this construction. Surprisingly, the resulting refined BPS structures do not seem to have appeared in the literature before.
\subsubsection{(Refined) Topological recursion and quantum curves}

Inspired by \cite{CE,CEM1,CEM2}, in our previous work we established a refined generalization of the topological recursion and corresponding quantization of genus zero spectral curves of degree two \cite{kidwai2023quantum}. Roughly speaking, refined topological recursion, which depends on a complex parameter $\beta\in\mathbb{C}^*$ (in fact, only on $\mathscr{Q}:=\beta^\frac12-\beta^{-\frac12}$), takes an algebraic curve and some additional data and produces a tower of multidifferentials $\omega_{g,n}$ ($g\in\frac12\mathbb{Z}_{\geq 0},n\in\mathbb{Z}_{>0}$) which may be considered as invariants of the curve. 


We may carefully assemble the $\omega_{g,n}$ together to obtain the \emph{topological recursion wavefunction} $\psi^{{\rm TR}}(x)$, which is the exponential of a formal $\hbar$-series. If there exists a differential operator on $X$ that annihilates the wavefunction of the form
\begin{equation}
\left(\hbar^{2}\dfrac{d^2}{dx^2}+q(x,\hbar)\hbar \dfrac{d}{dx}+r(x,\hbar)\right)\psi^{\rm \mathsmaller{TR}}(x)=0,
\end{equation} 
and whose semiclassical limit gives the equation for the spectral curve, then such an equation is called the \emph{quantum curve} associated to the spectral curve \cite{BE2}.

Quantizing curves in general is highly nontrivial, with progress made by a number of authors e.g. \cite{GS,BE1,Iwaki19}. Most relevant for us, for spectral curves of hypergeometric type, Iwaki-Koike-Takei \cite{iwaki2023voros,iwaki2019voros} were able to show existence of and obtain explicit expressions for quantum curves (in the unrefined case). Furthermore, they used the form of the quantum curves to obtain closed formulae for certain topological recursion invariants called \emph{free energies}. The refined analogues of these quantum curves were obtained in our previous work \cite{kidwai2023quantum}, and we call the corresponding quantum curves the \emph{(refined) quantum curves of hypergeometric type}.

Such differential operators have invariants. In particular, the so-called \emph{Voros coefficients} are certain formal $\hbar$-series of integrals (``quantum periods'') on the corresponding spectral curve, which play a fundamental role in WKB analysis and other applications. In general, it is difficult to find a closed formula for such coefficients. For unrefined quantum curves of hypergeometric type, an explicit computation of Voros coefficients was given in \cite{iwaki2019voros,iwaki2023voros}, following earlier progress \cite{TAKY:2008,publishedpapers/8600728}. One of the goals of the present work is to obtain an explicit formula for the Voros coefficients for the refined quantum curves of hypergeometric type, in the case without second order poles.


\subsection{Main results}

Our main results are twofold:
\begin{enumerate}
\item Explicit determination of Voros coefficients and free energies of the refined spectral/quantum Weber, Whittaker, degenerate Bessel, and Airy curves,
\item The observation that these can be expressed in terms of a refined BPS structure constructed from the corresponding quadratic differential, together with the quantum central charge.
\end{enumerate}
\noindent We conjecture these hold for all curves of hypergeometric type, though due to technical difficulties we were only able to confirm this for low values of $g$ in the case with second order poles.

These statements generalize the results of both \cite{iwaki2023voros,iwaki2019voros} and \cite{IWAKI2022108191} to the refined setting, and require several new ingredients. We propose values for $\Omega$ in Definition \ref{def:newRBPS}, some of which do not seem to have appeared before in the literature.
We also introduce the quantum correction to the central charge $\zbar:\Gamma\rightarrow \mathbb{C}$, essentially given by integrating along the new one-form $\omega_{\frac12,1}$ appearing as initial data in the refined topological recursion (see \S \ref{sec:quantumZ} for details).

For any $\gamma\in\Gamma$, let $\Omega_n(\gamma)$ denote the coefficient of $t^n$ in $\Omega(\gamma)$. Then our main theorem is:

\begin{theorem}\label{thm:mainintro}\, For the Weber, Whittaker, degenerate Bessel, and Airy refined spectral curves,
\begin{enumerate}
\item For half-integers $g>1$, the genus $g$ refined topological recursion free energy can be expressed in terms of the corresponding refined BPS structure $(\Gamma,Z,\Omega)$ and quantum correction $\zbar$ as

    \begin{equation}\label{eq:thm1aintro}    F_{g}={(-1)^{2g-2}}\sum_{\gamma\in\Gamma} \sum_{{ n \in \mathbb{Z}}} \dfrac{{\mathsf{B}_{2,2g}^{[n]}(\gamma)}}{4g(2g-1)(2g-2)}\Omega_n(\gamma)\left( 
{\dfrac{2 \pi i}{Z(\gamma)}} \right)^{2g-2}
\end{equation}
where we write
\begin{equation}
\mathsf{B}_{2,2g}^{[n]}(\gamma):=  B_{2,2g}{\left(\frac{\mathscr{Q}}{2}+\frac{\zbar(\gamma)}{2\pi i}+n\frac{\mathscr{Q}}{2}\,\Big|\,\beta^{\frac{1}{2}},-\beta^{-\frac{1}{2}}\right)}\end{equation}
with $B_{2,2g}$ the double Bernoulli polynomial of degree $2g$.
\medskip
\item  For integers $k\geq1$ and $\alpha$ a path which goes from $\infty_-$ to $\infty_+$ avoiding ramification points, the $k$-th Voros coefficient $V_{\alpha, k}$ is given explicitly by:


\begin{align}
\label{eq:thm1bintro}
V_{\alpha, k}
& =\frac{(-1)^{k+1}}{2}\sum_{\gamma \in \Gamma }\sum_{n\in\mathbb{Z}} 
{\langle{\gamma,\alpha}\rangle}{}\Omega_n(\gamma) \,\frac{\mathsf{B}_{1,k+1}^{[n]}(\gamma)}{k(k+1)}
 \, 
\left( \frac{2 \pi i}{Z(\gamma)} \right)^k.  
\end{align}
where $\langle\cdot,\cdot\rangle$ is the intersection pairing, and we write
{\begin{equation}
    \mathsf{B}_{1,k}^{[n]}(\gamma)=B_{1,k}\left(\frac{\beta^{\frac12}
    }{2}-\frac{\nu(\gamma)}{2\beta^{\frac12}}+\frac{{\zbar(\gamma)}}{2\pi i} +n\frac{{\mathscr{Q}}}{2}\,\Big|\,\beta^\frac12\right),
\end{equation}}
\noindent with $B_{1,k}$ the (single) Bernoulli polynomial of degree $k$.

\end{enumerate}
\end{theorem}

\noindent In fact, we claim that formula \eqref{eq:thm1aintro} continues to hold for all nine spectral curves in Figure \ref{fig:confdiag}: 

\begin{conjecture}
    For any refined spectral curve of hypergeometric type, \eqref{eq:thm1aintro} holds with the refined BPS structure and quantum correction to the central charge defined by the construction of \S\ref{sec:bpsfromqd}.
\end{conjecture}
An analogous conjecture should hold for the Voros coefficients, but we do not state it here for brevity of notation and to avoid a few subtleties. Although we are unable to prove the result for the cases with second order poles, we have checked \eqref{eq:thm1aintro} for low values of $g$ for the BPS structures given in Definition \ref{def:newRBPS}, and found perfect agreement.

We note that the BPS structure we have defined is not the only one that will make formula \eqref{eq:thm1aintro} hold. In particular, we show in Proposition \ref{prop:Blemma} that any other assignment $\Omega$ which preserves $\Omega_n+\Omega_{-n}$ for all $n$ will work. We may select our particular $\Omega$ by various means, such as palindromicity (when possible), or other hints in the literature, but a complete understanding remains to be discovered.

In order to prove our Theorem \ref{thm:mainintro}, we rely on a relation between the free energy and Voros coefficient, generalizing that of \cite{iwaki2023voros,iwaki2019voros} to the refined case, which is of independent interest:
\begin{proposition}\label{thm:diffrelnintro}
Suppose $\mathcal{S}^{{\bm \mu}}$ is either the Weber or Whittaker curve, and let $\alpha$ denote a path from $\infty_-$ to $\infty_+$ avoiding ramification points. We have:


\begin{equation}\label{eq:diffrelnintro}
V_{\alpha}=F\left(m-\frac{\nu-1}{2\beta^\frac12}\hbar\right)    - F\left(m-\frac{\nu+1}{2\beta^\frac12}\hbar\right) -\frac{1}{{\beta^{\frac12}}\hbar}\dfrac{\partial F_0}{\partial m}+\frac{\nu}{2\beta} \dfrac{\partial^2 F_0}{\partial^2 m}-\frac{1}{{\beta^{\frac12}}}\dfrac{\partial F_{\frac12}}{\partial m}
\end{equation}
 understood as a relation between formal series in $\hbar$.
\end{proposition}
\noindent It is in proving this relation that the main technical hurdle in the present paper appears, and we must carefully analyze the limiting behaviour of multi-integrals appearing in the definition of the topological recursion wavefunction (see Remark \ref{rem:crucial} and Lemma \ref{lem:continuous}). This was unnecessary in the unrefined case due to the simpler pole structure of $\omega_{g,n}$, and requires an argument involving the uniform boundedness of integrands. It is precisely the breakdown of this property which prevents our argument from applying in the presence of second-order poles, and resultingly, \eqref{eq:diffrelnintro} does not hold.

\subsection{Interpretation}

Several observations of this paper come as a surprise, and our result raises a number of questions. The most notable is that the value
\begin{equation}
\Omega(\gamma)=t+t^{-1}
\end{equation}
we associate to the Whittaker curve is new and unexpected, not appearing in Donaldson-Thomas theory nor supersymmetric field theory (though a similar structure appears in WKB analysis \cite{2000297,iwaki2014exact} describing the jump of Borel-resummed WKB solutions to the Whittaker equation in a special case). We were unable to find a satisfactory explanation for its appearance in these contexts.

We also note that a key role is played in our formula by the {quantum correction to the} central charge which appears naturally and automatically from the initial data of the topological recursion, but whose meaning is also not clear from either of those perspectives. We hope to understand what role it plays, if any, in the theory of stability and BPS structures in general.
 
It is natural to ask if the $\Omega(\gamma)$ we assign are in fact the refined Donaldson-Thomas invariants of some Calabi-Yau-3 triangulated category --- yet a cursory attempt suggests that they do not resemble virtual motives of any obvious space, which would be necessary for such an interpretation. Indeed, a category whose stability conditions are identified with quadratic differentials with simple poles was constructed by Haiden \cite{Haiden:2021qts} (although no higher order poles were allowed, in contrast to here). In that work, the refined BPS invariants were calculated, and in the unrefined limit were found to match with the BPS invariants predicted by one of the authors with K. Iwaki in \cite{IWAKI2022108191}. But in the refined case, if we compare the present paper's value for the Whittaker curve, we notice $\Omega(\gamma)=t+t^{-1}$ is certainly different from $\Omega(\gamma)=2$ (an analogous discrepancy arises in the case of the Legendre curve). Although they were not computed explicitly in the work of Bridgeland-Smith, it does not appear that the BPS invariants in that work agree with the ones we have given here, but rather with the ones observed by Haiden\footnote{This will be proved in \cite{kidwai:toappear}.}. Thus, a natural question is whether there is in fact some other CY3 category out there which correctly captures the structure appearing in the RTR free energy.

On the other hand, if they are not the refined BPS invariants of honest Donaldson-Thomas theory, it is natural to ask what they are. One possible answer would be to modify DT theory in an appropriate way; thinking along these lines is in fact natural from the quantum field theory point of view\footnote{We thank Andy Neitzke for pointing this out to us.}, which we elaborate on briefly in Remark \ref{rem:physics}. Whether such a ``multi-refinement'' would agree with our results is unclear, however, and further investigation is needed.

In any case, there appears to be one significant success predicted by our calculation. In search of an understanding of the invariants associated to second-order poles in our conjecture, upcoming work \cite{kidwai:toappear} shows that the corresponding invariant $\Omega(\gamma)=-t^{-1}$ actually \emph{does} appear as a BPS invariant of a CY3-triangulated category, under a natural identification of $t$. Thus, although some of our results remain mysterious, the prediction for second order poles indeed appears in the wild!
 

\subsection{Comments}
Let us make a few comments and mention some natural questions arising from our work.
\begin{itemize}
\item 
    The motivation for this work came from attempting to generalize the results of \cite{iwaki2023voros,iwaki2019voros,IWAKI2022108191,iwaki2023topological} to the refined setting. In particular, having established the formulas of this paper, we expect it should play a role in solving Barbieri-Bridgeland-Stoppa's so-called \emph{quantum Riemann-Hilbert problem} \cite{BBS} for corresponding refined BPS structures. It is natural to expect such considerations could hold more generally, in analogy with e.g. \cite{allegretti2021stability,bridgeland2023monodromy}, in which the uncoupled case was solved in the commutative setting. The relevant jump behaviour should also help clarify the true significance of the BPS structures observed here.


\item As in the unrefined case, we expect the formula (for an appropriate definition of $\Omega$) will continue to hold for any case in which the refined BPS structure is uncoupled. To check this in some examples, it would be necessary to e.g. generalize Y.M. Takei's construction of quantum curves in \cite{takei2020voros,doi:10.1142/9781800611368_0007,doi:10.1080/10652469.2021.1912038} to the refined setting.
\item Our formula for the Weber curve with parameter set to $\mu=1$ shows that $(-\frac{1}{\sqrt{\beta}})^{g}F_g$ coincides with the $s=0$ limit of the \emph{parametrized Euler characteristic} $\xi_{g}^{s}(\beta)$ of $s$-pointed real algebraic curves in \cite{goulden2001geometric}, generalizing the Harer-Zagier formula \cite{Zagier1986} (which can be obtained by the unrefined topological recursion \cite{iwaki2023voros}). It would be desirable to understand if there is a deeper relation between real or other kinds of algebraic curves and refined topological recursion \cite{iwaki2023voros}, and what the geometric meaning in this context of the parameter $\mu$ (introduced in \cite{kidwai2023quantum}) is.
\end{itemize} 

\subsubsection*{Acknowledgements}

We thank Tom Bridgeland, Andrea Brini, Ben Davison, Lotte Hollands, Kohei Iwaki, Gerg\H o Nemes, Andy Neitzke, and Nick Williams for helpful discussions and correspondence. 

Part of this research was conducted while visiting the Okinawa Institute of Science and Technology (OIST) through the Theoretical Sciences Visiting Program (TSVP)'s thematic program ``Exact Asymptotics: From Fluid Dynamics to Quantum Geometry''. Both authors partook in this program.

The work of OK was supported by the Leverhulme Trust's Research Project Grant ``Extended Riemann-Hilbert Correspondence, Quantum Curves, and Mirror Symmetry'', a JSPS Postdoctoral Fellowship for Research in Japan (Standard) and a JSPS Grant-in-Aid (KAKENHI Grant Number 19F19738). He also gratefully acknowledges the Simons Center for Geometry and Physics for hospitality and a stimulating environment during the program ``The Stokes Phenomenon and its Applications in Mathematics and Physics''.  The work of KO is in part supported by the Engineering and Physical Sciences Research Council under grant agreement ref.~EP/S003657/2, in part by the TEAM programme of the Foundation for Polish Science co-financed by the European Union under the European Regional Development Fund (POIR.04.04.00-00-5C55/17-00), and also in part by JSPS KAKENHI Grant Numbers 22J00102, 22KJ0715,
and 23K12968. 

\section{Review of refined topological recursion and quantum curves}

First, we recall the necessary definitions and results regarding the refined topological recursion \cite{kidwai2023quantum}. 

\subsubsection*{Multidifferentials}
Topological recursion involves a collection of geometric objects called multidifferentials defined on the product $(\mathcal{C})^k$ of a (say, compact) Riemann surface $\mathcal{C}$. 
Denote by $\pi_i:(\mathcal{C})^k\rightarrow \mathcal{C}$, $i=1,\ldots k$ the $i$th projection map. Then a (meromorphic) \emph{$k$-differential} is a meromorphic section of the line bundle
\begin{equation}\label{eq:multidif}
\pi_1^*(T^*\mathcal{C})\otimes \pi_2^*(T^*\mathcal{C})\ldots\otimes\pi_k^*(T^*\mathcal{C}). 
\end{equation}

We use the term \emph{multidifferential} to mean a $k$-differential without specifying $k$, and sometimes \emph{bidifferential} for the case $k=2$. Usually when writing multidifferentials in coordinates $z_1,\ldots z_k$ we simply omit the $\otimes$ symbol, and write e.g. $f(z_1,z_2) dz_1dz_2$. Integration and differentiation is defined factorwise in the obvious way. A multidifferential will be called \emph{symmetric} if it is invariant under permutation of variables:
\begin{equation}
\omega(p_1,\ldots,p_k)=\omega(p_{s(1)}, \ldots, p_{s(k)})
\end{equation}
for any permutation $s\in\mathfrak{S}_k$.

Note, on $\mathcal{C}=\mathbf{P}^1$, there is a distinguished bidifferential
\begin{equation}\label{eq:bidiff}
    B(z_1,z_2):=\dfrac{dz_1dz_2}{(z_1-z_2)^2}
\end{equation}
independent of the coordinate with a pole of biresidue $1$ along the diagonal and no poles elsewhere. $B$ is often called the fundamental bidifferential or the Bergman kernel.

\subsection{Refined Topological Recursion}

 We work in the setting in which refined topological recursion was established in our previous work \cite{kidwai2023quantum}. We omit a precise discussion of the assumptions, since all examples studied in this paper satisfy them. The main piece of data is a pair of functions $x,y:\mathcal{C}\rightarrow\mathbf{P}^1$ satisfying 
 \begin{equation}\label{eq:ysquaredQ}y^2 = Q(x)
 \end{equation}
 for rational $Q$ which is not a perfect square. In particular, our curves are canonically equipped with a globally defined involution $\sigma$ which is defined by exchanging sheets (viewing $x$ as a branched covering).

 Let us first fix some notation. We denote by $\mathcal{R}$ the set of ramification points of $x$, $\mathcal{P}$ the set of poles of $ydx$, and ${\mathcal{P}'}\subset\mathcal{P}$ the set of poles\footnote{The definition of ${\mathcal{P}'}$ here is slightly more general than one in \cite{kidwai2023quantum}, in which ${\mathcal{P}'}$ is defined as the set of unramified zeroes and poles of $y$ instead of $ydx$. All theorems and propositions shown in \cite{kidwai2023quantum} still hold as proven in \cite{Osu23-1}.} of $ydx$ that are not in $\mathcal{R}$. A ramification point $p\in\mathcal{R}$ is called \emph{ineffective} if $p\in\mathcal{R}\cap\mathcal{P}$ and \emph{effective} otherwise. We denote by $\mathcal{R}^*$ the set of effective ramification points. We assume that $ydx$ does not have zeroes away from ramification points --- this is a part of requirements of Assumption 2.7 in \cite{kidwai2023quantum}. Since ${\mathcal{P}'}$ does not contain ramification points, it admits a non-unique decomposition ${\mathcal{P}'}={\mathcal{P}'_{\!\mathsmaller{+}}}\cup\,\sigma({\mathcal{P}'_{\!\mathsmaller{+}}})$. We fix a choice of such a decomposition as part of our initial data, and to each $p\in {\mathcal{P}'_{\!\mathsmaller{+}}}$ we assign a parameter $\mu_p\in\mathbb{C}$. It is convenient to assemble them into one package as a (complex) divisor $D({\bm \mu})$ given by:

\begin{equation}
    D({\bm \mu}):=\sum_{p\in{\mathcal{P}'_{\!\mathsmaller{+}}}}\mu_p[p]\label{deta}
\end{equation}

\begin{definition}[\cite{kidwai2023quantum,Osu23-1}]\label{def:Rcurve}
A {\emph{genus zero degree two refined spectral curve}} $\mathcal{S}^{\bm \mu}$ is a quintuple \newline $\mathcal{S}^{\bm \mu}=(\mathcal{C},x,y,B,D({\bm \mu}))$ such that:
\begin{itemize}
    \item A compact Riemann surface $\mathcal{C}$ isomorphic to $\mathbf{P}^1$, 
    \item Meromorphic functions $x,y:\mathcal{C}\rightarrow \mathbf{P}^1$ satisfying \eqref{eq:ysquaredQ},
    \item $B$ is the unique symmetric bidifferential whose only pole is a double pole on the diagonal with biresidue 1,
    \item A choice of complex divisor $D({\bm \mu})$ supported at a choice of ${\mathcal{P}}_{\!\mathsmaller{+}}'$.
\end{itemize}
\end{definition}

We now define the refined topological recursion.

\begin{definition}[\cite{kidwai2023quantum,Osu23-1}]\label{def:RTR}
Fix $\mathscr{Q}\in\mathbb{C}$. Given a {genus zero degree two refined spectral curve} $\mathcal{S}^{\bm \mu}$, the \emph{refined topological recursion} is a recursive construction of an infinite sequence of multidifferentials $\omega_{g,n+1}$ for $2g,n\in\mathbb{Z}_{\geq0}$ with $2g+n\geq 0$ by the following formulae:
\allowdisplaybreaks[0]
\begin{equation}
    \omega_{0,1}(p_0):=y(p_0)dx(p_0), \qquad \omega_{0,2}(p_0,p_1):=-B(\sigma(p_0),p_1),
\end{equation}
\begin{equation}\label{eq:half1def}
    \omega_{\frac12,1}(p_0):=\frac{\mathscr{Q}}{2}\left(-\frac{dy(p_0)}{y(p_0)}+\sum_{p\in{\mathcal{P}'_{\!\mathsmaller{+}}}}\mu_p\eta_p(p_0)\right).
\end{equation}
\allowdisplaybreaks[1]

\noindent where $\eta_p(p_0):=\int_{\sigma(p)}^pB(p_0,\cdot)$, and for $2g-2+n\geq0$,
\begin{equation}
\omega_{g,n+1}(p_0,J)=-\left(\sum_{r\in\mathcal{R}}\underset{p=r}{\text{Res}}+\sum_{r\in\sigma(J_0)}\underset{p=r}{\text{Res}}+\sum_{r\in{\mathcal{P}'_{\!\mathsmaller{+}}}}\underset{p=r}{\text{Res}}\right)\frac{\eta_p(p_0)}{2\omega_{0,1}(p)}\!\cdot\text{Rec}_{g,n+1}^{\mathscr{Q}}(p,J),\label{eq:RTR}
\end{equation}
where
\begin{align}
\text{Rec}^\mathscr{Q}_{g,n+1}(p,J)=&\sum_{i=1}^n\frac{dx(p)\cdot dx(p_i)}{(x(p)-x(p_i))^2} \cdot \omega_{g,n}(p,\widehat{J}_i)+\sum^{*}_{\substack{g_1+g_2=g \\ J_1\sqcup J_2=J}} \omega_{g_1,n_1+1}(p,J_1) \cdot  \omega_{g_2,n_2+1}(p,J_2) \nonumber\\
&+\omega_{g-1,n+2}(p,p,J)+\mathscr{Q}\,dx(p)\cdot d_p\left(\frac{ \omega_{g-\frac12,n+1}(p,J)}{dx(p)}\right),\label{RecQ}
\end{align}

\noindent where $J:=(p_1,...,p_n)$ denotes a point in $(\mathcal{C})^n$, $\widehat{J}_i:=(p_1,\ldots \widehat{p_i},\ldots, p_n)$, $J_0:=(p_0,p_1\ldots,p_n)$, and $\sigma(J_0):=(\sigma(p_0),\sigma(p_1),...,\sigma(p_n))$; the $*$ in the sum means that we remove terms involving $\omega_{0,1}$, and $d_p$ is the exterior derivative with respect to $p$.
\end{definition}

Note that the above definition is slightly different from that in \cite{kidwai2023quantum}, in particular, the definition of $\omega_{0,2}$ and $\text{Rec}^\mathscr{Q}_{g,n+1}$. This is merely a difference in convention and they are equivalent as discussed in \cite{Osu23-1}. The following fundamental properties of the multidifferentials $\omega_{g,n}$ were shown in \cite{kidwai2023quantum}:

\begin{theorem}\label{theorem:RTR}
Let $\mathcal{S}^{\bm \mu }$ be a {genus zero degree two refined spectral curve}. For any $(g,n)$ with $2g,n\in\mathbb{Z}_{\geq0}$ and $2g+n-2\geq0$, the multidifferentials $\omega_{g,n+1}$ constructed from the refined topological recursion satisfy the following properties:
\begin{enumerate}
\itemsep2pt 
\item\label{RTR1} The multidifferential $\omega_{g,n+1}$ is symmetric;

\item\label{RTR2} All poles of $\omega_{g,n+1}$ (in any variable) lie in $\mathcal{R}^*\cup\sigma(J_0)$;

\item\label{RTR3} At any $ o\in \mathcal{C}$, $\omega_{g,{n+1}}$ is residue-free in the first (thus, any) variable:
\begin{equation}
    \underset{p=o}{{\rm Res}}\,\omega_{g,n+1}(p,J)=0;
\end{equation}

\item \label{RTR4}  For $n>0$ we have
\begin{equation}
(2g+n-2)\,\omega_{g,n}(J)=-\left(\sum_{r\in\mathcal{R}^*}\underset{p=r}{{\rm Res}}+\sum_{r\in\sigma(J)}\underset{p=r}{{\rm Res}}\right)\Phi(p)\cdot\omega_{g,n+1}(p,J).\label{Rinverse}
\end{equation}

\end{enumerate}

\end{theorem}

In addition to these fundamental properties, the following important lemma, also proved in \cite{kidwai2023quantum}, will be needed to establish the behaviour of various integrals below.
\begin{lemma}[Global loop equations]\label{lem:GLE} Given a {genus zero degree two refined spectral curve}, for any $(g,n)$ with $g,n\in\mathbb{Z}_{>0}$ and $2g+n -2>0$, we have:
    \begin{equation}
    \omega_{g,n+1}(p_0,J)=-\frac{{\rm Rec}^{\mathscr{Q}}_{g,n+1}(p_0,J)}{2\omega_{0,1}(p_0)} +\sum_{i=1}^n d_{p_i}\left(\frac{\eta_{p_i}(p_0)}{2\omega_{0,1}(p_i)}\cdot\omega_{g,n}(J)\right).\label{pole2}
\end{equation}
\end{lemma}

The fundamental object we are interested in is the so-called free energy. Thanks to the residue-free property \eqref{RTR3}, we can define this ``$\omega_{g,0}$'' for $g>1$ as follows (we will give the definition for $g=0,\frac12,1$ later):
\begin{definition}\label{def:rFg}
For $g>1$, the \emph{genus $g$ free energy $F_g$} of the refined topological recursion is defined by
\begin{equation}
F_g=\frac{1}{2-2g}\sum_{r\in\mathcal{R}^*}\underset{p=r}{\text{Res}}\,\Phi(p)\!\cdot\omega_{g,1}(p),\label{RTRFg}
\end{equation}
where $\Phi$ is any primitive of $ydx$.
\end{definition}

\subsection{Refined quantum curves and Voros coefficients}

\subsubsection{(Refined) quantum curves}

With $\omega_{g,n}$ on hand, we may carefully assemble them into a natural object. The \emph{topological recursion wavefunction} $\psi^{{\rm TR}}(x)$ is defined as:
\begin{equation}\label{eq:wavefn}
    \psi^{\rm TR}(x) := \exp \left( \sum_{g\in \frac{1}{2}\mathbb{Z}_{\geq 0},n>0} \dfrac{\hbar^{2g-2+n}}{\beta^{n/2}} \dfrac{1}{n!} \int_{ D(z;{ {\bm \nu}})}\!\!\!\ldots\int_{D(z;{ {\bm \nu}})}\omega_{g,n}\right)
\end{equation}
which should be understood as the exponential of a formal $\hbar$-series\footnote{Since our definition of $\omega_{0,2}$ is slightly different from the usual convention, we do not need a shift for $\omega_{0,2}$ in the definition of the wave function which commonly appears in literature (e.g. \cite{kidwai2023quantum,BE2})}. Here, we integrate from the rightmost variable to the leftmost, along a complex divisor
\begin{equation}
    D(z)=D(z,{\bm\nu}):=\sum_{p\in{\mathcal{P}}}\nu_p[p], \qquad \sum_{p\in\mathcal{P}}\nu_p=1,
\end{equation}
where integration of any $\omega$ is defined by
\begin{equation}
    \int_{D(z)} \omega \,\,:= \sum_{p\in\mathcal{P}}\nu_p\int_{p}^z \omega 
\end{equation}

The definition of $\psi^{\rm TR}$ relies on the following lemma proved in \cite{kidwai2023quantum}:
\begin{lemma}\label{lem:residue3}
Given a genus zero degree two refined spectral curve, let $p_i \in\mathcal{C}$ for $i=0,\ldots n$ and suppose that $p_i\not\in\mathcal{P}$, and $p_i\neq \sigma(p_j)$ for any $i,j$. Suppose furthermore that lower endpoints $q_i\in {\mathcal{P}}$ are chosen. Then for all $2g+n\geq2$, the integral
\begin{equation}
   \int^{p_0}_{q_0}\cdots\int^{p_{n}}_{q_{n}}\omega_{g,n+1}\label{residue3}
\end{equation}
exists and defines a meromorphic function $F_{g,n+1}$ on $\Sigma$ with respect to each $p_i$, regular at $p_i=p_j$ for any $i,j$, and is independent of the path of integration.

\end{lemma}

\begin{remark}[Crucial remark] \label{rem:crucial}
      This integral is in fact not guaranteed to be continuous in $p_i$ at the excluded points, even if it exists. Thus, we have said that \eqref{residue3} \emph{defines} a meromorphic function $F_{g,n+1}$, with the understanding that it is given by a rational expression whose value is given by the integral \eqref{residue3} on a dense open subset of $\mathcal{C}^{n+1}$. We note that due to this, for some special points the integral signs must strictly speaking be interpreted as a limit of the meromorphic function $F_{g,n+1}$. This subtle point turns out to cause a crucial difficulty which separates the examples treated in this paper from the examples with second order poles, and will matter in \S\ref{sec:main} below when we consider the Voros coefficient. We treat the Weber and Whittaker curves in this paper precisely because in this case, the integral actually \emph{is} continuous where needed, and the wavefunction can be interpreted as the genuine integrals in \eqref{eq:wavefn}.
\end{remark}

Roughly speaking, a \emph{quantum curve} is a differential operator on $X$ which annihilates $\psi^{\rm TR}$:
\begin{equation}
    \mathcal{D}_\hbar(\bm \nu)\psi^{\rm TR} = 0
\end{equation}
More precisely, let $\mathcal{S}^{\bm \mu}=(\mathcal{C},x,y,B,D({\bm \mu}))$ be a (smooth) genus zero degree two refined spectral curve. Given a divisor $D(z)$ as above,
a quantum curve is an ODE with formal solution $\psi^{\rm \mathsmaller{TR}}(x) = \varphi^{\rm \mathsmaller{TR}}(z(x))$:
\begin{equation}
\left(\hbar^{2}\dfrac{d^2}{dx^2}+q(x,\hbar)\hbar \dfrac{d}{dx}+r(x,\hbar)\right)\psi^{\rm \mathsmaller{TR}}(x)=0,
\end{equation}  
for any choice of inverse $z(x)$, such that $q(x,\hbar)$, $r(x,\hbar)$ are (finite) formal $\hbar$-series of rational functions in $x$
\begin{equation}
    q(x,\hbar)=q_0(x)+\hbar q_1(x), \quad r(x,\hbar) = r_0(x) + \hbar r_1(x) + \hbar^2 r_2(x).
\end{equation}
and the semiclassical limit ($\beta^{\frac12}\hbar\frac{d}{dx}\rightarrow y$ followed by $\hbar\rightarrow 0$) recovers the algebraic equation satisfied by the functions $x$ and $y$.

\subsection{Spectral curves of hypergeometric type}
In \cite{kidwai2023quantum} the refined topological recursion for genus zero degree two refined spectral curves was formulated and established, but in the present paper we are concerned with some explicit examples. In practice, these arise from a quadratic differential, that is, a meromorphic section of $\omega_X^{\otimes 2}$, where $X$ is a compact Riemann surface. Any such $\varphi$ has the usual notion of poles, zeroes, and their orders by passing to a local coordinate chart. The set of poles is denoted $P$, and the set of \emph{turning points} (simple zeroes and simple poles) is denoted $T$.

Given a (reasonable) quadratic differential $\varphi$, we may construct the so-called \emph{spectral cover}. First, we form
\begin{equation}
    \Sigma=\left\{\lambda\,|\, \lambda^2=\pi^*\varphi\right\}\subset T^*X
\end{equation}

\noindent which is a smooth subvariety of the cotangent bundle. We equip $\Sigma$ with the restrictions of the projection map $\pi:T^*X \rightarrow X$ and the tautological one-form $\lambda$. If $x$ denotes a coordinate on $X$, and $y$ the corresponding fibre coordinate, we may write this as
\begin{equation}\label{eq:hyperelliptic}
    y^2=Q(x).
\end{equation}
as a local equation in $T^*X$. The corresponding compact Riemann surface is denoted $\overline{\Sigma}$, and $\pi$, $\lambda$ extend to a branched double covering and meromorphic one-form, respectively; we refer to this data $(\overline{\Sigma},\pi,\lambda)$ as the spectral cover.

Although this construction is very general, in the present paper, we will be concerned only with the following four examples:\medskip
\begin{align}
    &\varphi_{\rm Web}=\left(\frac{x^2}{4}-m\right)dx^2&\hspace{-1cm} &\varphi_{\rm Whi}=\frac{x+4m}{4x}dx^2\\ &\varphi_{\rm dBes}= \frac{1}{4x}dx^2 &\hspace{-1cm} &\varphi_{\rm Ai}= x dx^2
\end{align}
where $x$ is a coordinate on $X=\mathbf{P}^1$ and $m\neq 0$ is a complex parameter. The refined spectral curve in the sense of topological recursion is then given by $\overline{\Sigma}$ together with the functions $x:=\pi$ and $y:=\lambda/dx$ (and the canonical choice \eqref{eq:bidiff} for $B$). The set of poles of $\lambda$ is then exactly $\mathcal{P}=\pi^{-1}(P\setminus T)$. It is easy to verify that the curves above together with a choice of divisor $D({\bm \mu}) = \sum_{p\in\mathcal{P}'}\mu_p [p]$ are all examples of genus zero degree two refined spectral curves, so that the results of \cite{kidwai2023quantum} may be applied. There are several other examples, corresponding to those in Figure \ref{fig:confdiag} in which the differential also has second-order poles, written out explicitly in e.g. \cite{kidwai2023quantum}. We call these examples the \emph{(refined) spectral curves of hypergeometric type}, although in this paper we only consider the four above, which we call the Weber, Whittaker, degenerate Bessel, and Airy curves, respectively.

We note that the Weber and Whittaker curves come in a one-dimensional family parametrized by a complex number $m\neq 0$, called the \emph{mass}. The Airy and degenerate Bessel curves do not contain any parameters. We will label points in $\mathcal{P}$ by their image under $x$ and a subscript $\pm$ with the convention that
\begin{equation}\label{eq:signconv}
     \underset{p=\infty_\pm}{{\rm Res}}ydx(p)=\pm m
\end{equation}
Furthermore, the the Weber and Whittaker refined spectral curves come with one parameter $\mu$ associated to the choice $\mathcal{P}'_{\!\mathsmaller{+}
}=\{\infty_+\}$.

\subsubsection{Refined quantum curves of hypergeometric type}
One of our goals is to determine certain invariants of differential equations obtained via the refined topological recursion. These equations were determined in \cite{kidwai2023quantum}, and can be summarized as follows:
\begin{theorem}[\cite{kidwai2023quantum}]\label{theorem:main}
For all spectral curves of hypergeometric type, the refined quantum curve exists and can be explicitly computed. In particular, we have:\newline
Weber:
\begin{equation}
\left(
\epsilon_1^2 \dfrac{d^2}{dx^2}  -\frac{x^2}{4  }  +m+ \frac{ \epsilon_2\nu-(\epsilon_1\!\!+\!\epsilon_2)\mu}{2} \right) \psi^{\rm TR} (x) =0
\end{equation}
Whittaker:\footnote{We have made a gauge transformation to eliminate the $\frac{d}{dx}$ term in comparison with \cite{kidwai2023quantum}.}


\begin{equation}
   \left( \epsilon_1^2\frac{d^2}{dx^2}-\frac14-\frac{m}{x}-\frac{\epsilon_1\nu-(\epsilon_1+\epsilon_2)\mu}{2x}+\frac{\epsilon_1^2}{4x^2}-\frac{(\epsilon_1+\epsilon_2)^2}{4x^2}\right)\psi^{\rm TR}(x)=0
\end{equation}
 Degenerate Bessel:
\begin{equation}
    \left(\epsilon_1^2\frac{d^2}{dx^2}-\frac{1}{4x}\right)\psi^{\rm TR}(x)=0
\end{equation}
Airy:
\begin{equation}
    \left(\epsilon_1^2\frac{d^2}{dx^2}-x\right)\psi^{\rm TR}(x)=0
    \end{equation}
where we have written $\nu=\nu_{\infty_+}\!\! - \nu_{\infty_-}$, $\epsilon_1=\hbar\beta^{\frac12}$, $\epsilon_2=-\hbar \beta^{-\frac12}$. 
\end{theorem}
\noindent Expressions for other hypergeometric type curves can be found in \cite{kidwai2023quantum}.

\subsubsection{Cycles on the spectral curve}
\label{section:cycle-path}
Voros coefficients are certain invariants of the quantum curve which are obtained by integrating a certain (formal $\hbar$-series of) one forms along various cycles. We will not give a general definition to avoid introducing excess notation, and simply allow $\gamma$ to denote the generator of $\Gamma$ such that
\begin{equation}
    \int_\gamma ydx = 2\pi i m
\end{equation}
The other important (relative) cycle will be given by a path $\alpha$ on $\widetilde{\Sigma}$ which goes from $\infty_-$ to $\infty_+$ avoiding ramification points, considered as a cycle $\alpha \in H_1(\overline{\Sigma},\pi^{-1}\{\infty\},\mathbb{Z})$. Though only the class of the path (for which there is only one choice) matters, we will often use $\alpha$ to denote some chosen representative (avoiding the ramification points). It is clear that both $\gamma$ and $\alpha$ are anti-invariant with respect to the involution $\sigma$.

\subsubsection{Voros coefficients}

Equipped with the quantum curve, we may ask for its invariants. One natural type of invariant is a formal series in $\hbar$ called the \emph{Voros coefficient}, sometimes known as a \emph{quantum period}, since its leading term is the period of $ydx$ on the spectral curve.

Let $T$ denote the log-derivative of $\varphi=\psi \circ x$. We consider its ``odd part":
\begin{equation}
T_{\rm odd}(z, \hbar)dz := 
\frac{T(z,  \hbar) - \sigma^\ast T(z,  \hbar)}{2}dz,
\end{equation}
where $\sigma$ is the covering involution of $\overline{\Sigma}$. This is a formal series of one-forms whose coefficients are meromorphic on $\overline{\Sigma}$, satisfying 

\begin{equation} \label{anti-invariance-dSodd}
\sigma^\ast T_{\rm odd}(z, \hbar) 
= - T_{\rm odd}(z, \hbar).
\end{equation}

To avoid divergent integrals, we denote the same formal series with its first two terms truncated as:
\begin{equation}T^{\geq 1}_{
\rm odd}(z,\hbar):=
T_{\rm odd}(z, \hbar) - 
\dfrac{T^{(-1)}_{\rm odd}(z)}{\hbar} - T^{(0)}_{\rm odd}(z).
\end{equation}

\begin{definition}
~
\begin{itemize}
\item 
The \emph{Voros coefficient for the cycle $\gamma$} of the quantum curve is a formal series of $\hbar$ defined by the term-wise integral:
\begin{equation}  \label{def-Voros-closed}
V_{\gamma}(\hbar)
:= \oint_{\gamma} T_{\rm odd}(z, \hbar)dz.
\end{equation}

\item 
The \emph{Voros coefficient for the path $\alpha$} of the quantum curve is a formal series of $\hbar$ defined by the term-wise integral:
\begin{equation}\label{def-Voros-relative}
V_\alpha(\hbar):= \int_\alpha T^{\geq 1}_{\rm odd}(z,\hbar)dz
\end{equation}
\end{itemize}
For the degenerate Bessel and Airy equations, no Voros coefficients are defined since the homology groups are trivial.
\end{definition}
Since the integrand has poles only at ramification points and is anti-invariant, $V_\alpha$ is well-defined independent of the choice of precise path, as long as the relative homology class of $\alpha$ in $H_1(\overline{\Sigma},\pi^{-1}(\infty),\mathbb{Z})$ remains fixed (and similarly for $V_\gamma$).
\subsubsection{Variational formula}
We will show below a relation between the free energy $F_g$ and and Voros coefficients $V_{\beta}$, as a generalisation of results shown in \cite{iwaki2023voros,iwaki2019voros}. The crucial ingredient is the so-called \emph{variational formula} proved in \cite{Osu23-2} as a refined analogue of the formula in \cite{EO}. 

Let us first fix some notation. For simplicity, we only state the formula for the Weber and Whittaker curve, though analogous results hold for other hypergeometric curves\footnote{There is no variational formula for degenerate Bessel and Airy, since there is no mass parameter.}. We denote by $\infty_\pm\in{\mathcal{P}}$ the (unique two) poles of $\lambda$, which satisfy $\pi(\infty_{\pm})=\infty$, and the sign convention in \eqref{eq:signconv} We denote by $\alpha$ any oriented path which begins at $\infty_-$ and ends at $\infty_+$, avoiding ramification points. Then, $F_0,F_\frac12,F_1$ for the Weber and Whittaker curve are respectively defined as follows.
\begin{definition}\label{def:Funstable}
    Given a refined spectral curve of the Weber  or Whittaker type, $F_0,F_\frac12,F_1$ are defined as a solution of the following differential equations:
    \begin{align}
        \frac{\partial^3 F_{0}}{\partial  m^3}=\int_{\alpha}\int_{\alpha}\int_{\alpha}\omega_{0,3},\quad  \frac{\partial^2F_{\frac12}}{\partial  m^2}=\int_{\alpha}\int_{\alpha}\omega_{\frac12,2},\quad  \frac{\partial F_{1}}{\partial  m}=\int_{\alpha}\omega_{1,1}.
    \end{align}
\end{definition}
Note that $F_0,F_\frac12,F_1$ comes with a polynomial ambiguity of $m$ of degree $2-2g$ which can be chosen arbitrarily. The variational formula states that an analogous relation holds for all $\omega_{g,n}$:

\begin{proposition}[\cite{Osu23-2}]\label{theorem:variation} Let $F_g$ denote the free energy of the refined topological recursion for the Weber or Whittaker curves. For any $2g+n-2\geq 1$ we have the following relation:
\begin{equation}
\frac{\partial^n F_{g}}{\partial  m^n}=\int_\alpha\cdots\int_\alpha\,\omega_{g,n}.\label{Fvariation}
\end{equation}
\end{proposition}

For the Weber curve, $F_0,F_\frac12,F_1$ can be explicitly computed as
\begin{align}
    F_0&=\frac{1}{2} m^2 \log (m)-\frac{3 m^2}{4},\label{F0Web}\\
F_{\frac12}&=\frac12\mu\mathscr{Q} m\log m - \frac12\mu\mathscr{Q}m,\label{F1/2Web}\\
 F_1&=-\frac{2+(1-3\mu^2)\mathscr{Q}^2}{24}\log m,\label{F1Web}
\end{align}
where we made a convenient choice for the polynomial ambiguity. Similarly, for the Whittaker curve, we have:
\begin{align}
    F_0&=m^2 \log (m)-\frac{3 m^2}{2},\\
F_{\frac12}&=\mu\mathscr{Q} m\log m - \mu\mathscr{Q}m,\\
 F_1&=-\frac{2+(2-3\mu^2)\mathscr{Q}^2}{12}\log m,
\end{align}
For the Airy curve and the degenerate Bessel curve, there is no parameter in the spectral curve, and we set $F_0=F_\frac12=F_1=0$ in this case.

\section{Refined BPS structures from quadratic differentials}
\subsection{Refined BPS structures}
Our main result will show that the free energies and Voros coefficients can be written in terms of a refined BPS structure associated to the quadratic differential, together with an extra $\hbar$-dependent piece of data. The former is defined as follows:
\begin{definition}
 A \emph{refined BPS structure} is a tuple $(\Gamma,Z,\Omega)$ of the following data:
\begin{itemize}
\item a free abelian group of finite rank $\Gamma$ equipped with an antisymmetric pairing 
\begin{equation}
\langle \cdot , \cdot \rangle : \Gamma \times \Gamma \rightarrow \mathbb{Z},
\end{equation}
\item a homomorphism of abelian groups $Z:= \Gamma \rightarrow \mathbb{C}$, and
\item a map of sets $\Omega : \Gamma \to {\mathbb Q}[t,t^{-1}]$,
\begin{equation}
    \Omega(\gamma)=\sum_{n\in \mathbb{Z}}\Omega_n(\gamma)\cdot t^n
\end{equation}
where $t$ is a formal symbol,
\end{itemize}
satisfying the conditions
\begin{itemize}
\item \emph{Symmetry}: $\Omega(\gamma)=\Omega(-\gamma)$ for all $\gamma \in \Gamma$, and $\Omega(0)=0$.
\item \emph{Support property}: for some (equivalently, any) choice of norm $\|\cdot\|$ on 
$\Gamma \otimes \mathbb{R}$, there is some $C>0$ such that
\begin{equation} \label{support-property}
\Omega(\gamma)\neq  0 \implies |Z(\gamma)|>C\cdot \|\gamma\|.
\end{equation}
\end{itemize}
We call $\Gamma$ the \emph{charge lattice}. The homomorphism $Z$ is called the \emph{central charge}. 
The  $\Omega(\gamma)$ are called the \emph{(refined) BPS invariants} of the refined BPS structure. 

\end{definition}

Note that when $\Omega$ is concentrated entirely at $n=0$ (that is, $\Omega_n=0$ for all $n\neq0$ so $\Omega$ is valued in $\mathbb{Q}$), we recover the ordinary notion of BPS structure \cite{Bri19}. 
We refer to an element $\gamma \in \Gamma$ which has 
nonzero BPS invariant, $\Omega(\gamma)\neq 0$, as an \emph{active class}. We note that the central charge $Z(\gamma)$ for an active class $\gamma$ is always nonzero 
due to the support property. 




We can also consider certain classes of refined BPS structures with particularly nice properties.
A refined BPS structure $(\Gamma, Z, \Omega)$ is said to be
\begin{itemize} 
\item 
\emph{finite} if there are only finitely many active classes,
\item
\emph{uncoupled} if $\langle \gamma_{1},\gamma_{2}\rangle=0$ holds whenever 
$\gamma_{1},\gamma_{2}$ are both active classes (otherwise, it is \emph{coupled})
\item \emph{palindromic} if $\Omega_n(\gamma)=\Omega_{-n}(\gamma)$ for all $\gamma\in\Gamma$, $n\in \mathbb{Z}$.
\item
\emph{integral} if $\Omega_n$ take values in ${\mathbb Z}$.

\end{itemize}

All the examples related to hypergeometric type curves are associated finite, uncoupled, integral refined BPS structures.

 In anticipation of its role later, we also introduce the notion of an $\hbar$-correction to the central charge.
\begin{definition} Given a BPS structure $(\Gamma,Z,\Omega)$, a \emph{quantum correction to the central charge} (or \emph{quantum correction} for short) will be a choice of homomorphism 
    \begin{equation}
        \zbar:\Gamma\rightarrow \mathbb{C}\llbracket \hbar\rrbracket
    \end{equation}
\end{definition}
We call the combination $\widehat{Z}:=Z+\hbar\,{\zbar}$ the \emph{quantum central charge}.

\subsection{BPS structures from quadratic differentials}\label{sec:bpsfromqd}


In previous joint work of the first named author \cite{IWAKI2022108191}, a generalized Gaiotto-Moore-Neitzke method was used to construct BPS structures from the quadratic differentials in Figure \ref{fig:confdiag}. Here we briefly recall the construction based on spectral networks associated with these quadratic differentials, before generalizing to the refined case and proposing a modified construction suited to our purposes. We note that the construction of BPS structures given here can be generalized significantly \cite{GAIOTTO2013239}, though there are technical issues involved.

Given a meromorphic quadratic differential $\varphi$, recall that we have the spectral cover
\begin{equation}
    \Sigma=\{\lambda\,|\, \lambda^2=\pi^*\varphi\}\subset T^*X.
\end{equation}
We will furthermore ``fill-in'' (partial compactification) the points corresponding to simple poles of $\varphi$ and denote the resulting curve by $\widetilde{\Sigma}$. We equip $\widetilde{\Sigma}$ with the restrictions of the projection map $\pi:T^*X \rightarrow X$ and the tautological one-form $\lambda$ as before. 

\subsubsection{Lattice and central charge} 
Recall that a BPS structure requires a charge lattice $\Gamma$ and a central charge $Z$. Following \cite{GAIOTTO2013239,bridgeland2015quadratic}, we define:
\begin{definition} \label{def:central-charge}
The charge lattice $\Gamma$ associated to $\varphi$ is the sublattice 
of $H_1(\widetilde{\Sigma}, {\mathbb Z})$ defined by
\begin{equation}
\Gamma := \{ \gamma \in H_1(\widetilde{\Sigma}, {\mathbb Z}) 
~|~ \sigma_\ast \gamma = - \gamma \}
\end{equation}
equipped with the intersection pairing $\langle \cdot , \cdot \rangle : \Gamma \times \Gamma \to {\mathbb Z}$. The central charge $Z : \Gamma \to {\mathbb C}$ is the period integral: 
\begin{equation}
Z(\gamma) := \oint_{\gamma} \lambda.
\end{equation}
\end{definition}

 The lattice $\Gamma$ is called the hat-homology group in \cite{bridgeland2015quadratic}. For the Weber and Whittaker curves it is of rank $1$ generated by $\gamma$, and for degenerate Bessel and Airy it is trivial (rank $0$). Since $\overline{\Sigma}$ is of genus zero, the intersection pairing $\langle \cdot , \cdot \rangle$ is always trivial in our examples. As a result, the BPS structures we will obtain are a fortiori uncoupled. 

\subsubsection{Spectral networks}
To define the function $\Omega$, we use the notion of a \emph{(WKB) spectral network} \cite{Gaiotto:2012rg} associated to the quadratic differential $\varphi$\footnote{A more general construction can be made for arbitrary tuples of meromorphic $k$-differentials, but it is more complicated. We only consider BPS structures arising from quadratic differentials in this paper, so we omit the details.}. 
In this case, it is essentially synonymous with the notion of \emph{Stokes graph} (c.f., \cite{kawai2005algebraic}) of the corresponding quantum curve, which controls the summability properties of WKB solutions of certain Schro\"dinger-type equations \cite{koikeschafke,NEMES2021105562,nikolaev2023exact}.

We define a \emph{trajectory} of $\varphi$ to be any maximal curve on $X$ along which
\begin{equation}
\mathrm{Im}\, e^{-i \vartheta}\int^{x}{ \sqrt{\varphi}} = {\rm constant}.
\end{equation}
\noindent is satisfied. The general structure of trajectories of quadratic differentials is described in detail in the classic book by Strebel \cite{strebel1984quadratic} (see also \cite{bridgeland2015quadratic}), which we refer to for details. We are primarily interested in trajectories with at least one endpoint on a turning point --- we call these \emph{critical trajectories}. The set of (oriented) critical trajectories is called the \emph{spectral network $\mathcal{W}_\vartheta({\varphi})$ of $\varphi$ at phase $\vartheta$}.

A fundamental and difficult in general problem is to understand and classify the trajectories associated to a given $\varphi$. This was determined for the examples in this paper in \cite{IWAKI2022108191}, but since we will only really use the BPS structure which is obtained as output, we do not delve into details and examples of spectral networks; we refer the reader to \cite{IWAKI2022108191,GAIOTTO2013239,bridgeland2015quadratic,strebel1984quadratic} for more details.

\subsubsection{Degenerate spectral networks and saddle trajectories}
\label{subsection:saddle-and-BPS-indices}

We say the spectral network $\mathcal{W}_\vartheta(\varphi)$ is \emph{degenerate} if it contains a {\em saddle trajectory} (i.e., a critical trajectory both of whose endpoints are turning points) or {\em closed trajectory} of phase $\vartheta$. 
%
%
In the physics of 4d $\mathcal{N}=2$ QFTs, they signal the presence of certain BPS states in the spectrum of the theory \cite{GAIOTTO2013239}; from a mathematical point of view, they correspond (in certain cases) to stable objects in a 3-Calabi-Yau category associated to a quiver with potential determined by $\varphi$ \cite{bridgeland2015quadratic}. 

Let us summarize the possible types of saddle trajectories which can appear in degenerate spectral networks arising when $\varphi$ is of hypergeometric type\footnote{Although there is a fifth type of saddle trajectory (which forms the boundary of a ``non-degenerate ring domain'' and the associated class $\gamma$ gives $\Omega(\gamma) = -2$ as its unrefined BPS invariant) in general, we omit it here since it does not appear for hypergeometric type $\varphi$.}: 
\begin{itemize}
\item 
A {\em type I}\, saddle connects two distinct simple zeros of $\varphi$. 
\item 
A {\em type II}\, saddle connects a simple zero and a simple pole of $\varphi$.
\item 
A {\em type III}\, saddle connects two distinct simple poles of $\varphi$.
\item 
A {\em type IV}\, saddle is a closed curve which forms the boundary of a degenerate ring domain (i.e., a maximal region foliated by closed trajectories around a second order pole of $\varphi$).
\end{itemize}

We note that type IV trajectories do not appear for the Weber, Whittaker, degenerate Bessel and Airy differentials, and in fact they are precisely the case in which additional technical difficulties arise in extending our result.

\subsubsection{Definition of $\Omega$}

Given a degenerate spectral network containing a saddle trajectory of type I, II, or III, we can associate (up to sign) to that saddle a homology class $\gamma_{\rm BPS} \in \Gamma$ represented by the closed cycle (up to its orientation) on $\widetilde{\Sigma}$ obtained as the pullback by $\pi : \widetilde{\Sigma} \to X$ of the saddle trajectory. On the other hand, for a degenerate spectral network with a type IV saddle, we define  $\gamma_{\rm BPS} = \gamma_{s_+} - \gamma_{s_-} \in \Gamma$ (up to orientation) where $s$ is the second order pole of $\varphi$ in the degenerate ring domain surrounded by the type IV saddle. We refer to homology classes $\gamma_{\rm BPS} \in \Gamma$ defined in this way as \emph{BPS cycles}. Degenerate spectral networks only appears when the phase $\vartheta$ coincides with the argument of the central charge $Z(\gamma_{\rm BPS})$ of a certain BPS cycle.



Now we recall the computation of the BPS structures introduced in \cite{GAIOTTO2013239, bridgeland2015quadratic, IWAKI2022108191}. Given any $\varphi$ of hypergeometric type, the BPS structure was determined in \cite{IWAKI2022108191}. The results are summarized in Table \ref{table:BPS-str-HG} below; we refer the reader to \cite{IWAKI2022108191} for additional details.

\begin{theorem}[{\cite[\S 4]{IWAKI2022108191}}]
Let $\varphi=Q(x)dx^2$ be any quadratic differential of hypergeometric type. The BPS structure determined by $\varphi$ is the one given in Table \ref{table:BPS-str-HG}. In particular, the BPS structure obtained is finite, integral and uncoupled.
\end{theorem}

\begin{table}[h]
\begin{center}
\begin{tabular}{cccc}\hline
Spectral curve & $\pm\gamma_{\rm BPS}$ & $\pm Z(\gamma_{\rm BPS})/2\pi i$ & $\Omega(\gamma_{\rm BPS})$ 
\\\hline
\vspace{1mm}
\parbox[c][2.5em][c]{0em}{}
Gauss 
& $ \pm (\gamma_{0_+} + \gamma_{1_s } + \gamma_{\infty_{s'} }) $ 
\hspace{+.5em} $(s, s' \in \{ \pm \})$ 
&$ \pm ( m_0+s\, m_1+ s' m_\infty)$
&  $1$
\\[-.5em] 
& $\pm (\gamma_{p_+} - \gamma_{p_-})$
\quad $(p \in \{0,1,\infty \})$
&$\pm \, 2 m_p$ 
& $-1$\vspace{1mm}
\\[+.3em]\hline
\parbox[c][2.5em][c]{0em}{}
\begin{minipage}{.15\textwidth}
\begin{center} Degenerate Gauss \end{center}
\end{minipage}
& $\pm (\gamma_{1_+} + \gamma_{\infty_s})$ 
\quad $(s \in \{ \pm \})$
&$\pm (m_1+s\, m_\infty)$
&  $2$
\\[-.em]
& $\pm (\gamma_{p_+} - \gamma_{p_-})$
\quad $(p \in \{1,\infty \})$
&$\pm 2 m_p$ 
& $-1$ \vspace{1mm}
\\[+.3em]\hline
\parbox[c][2.5em][c]{0em}{}
Kummer 
& $\pm (\gamma_{0_+} +  \gamma_{\infty_s})$ 
~ $(s \in \{ \pm \})$
& $\pm (m_0 + s \, m_\infty)$
&  $1$
\\[-.em]
& $\pm (\gamma_{0_+} - \gamma_{0_-})$
&$\pm 2 m_0$ 
& $-1$ \vspace{1mm}
\\[+.3em]\hline
\parbox[c][2.5em][c]{0em}{}
Legendre 
& $\pm \gamma_{\infty_+}$
&$\pm m_\infty$
& $4$ 
\\[-.em]
& $\pm (\gamma_{\infty_+} -  \gamma_{\infty_-})$
& $\pm 2 m_\infty$ 
& $-1$ \vspace{1mm}
\\[+.3em]\hline
\parbox[c][2.5em][c]{0em}{}
Bessel 
&
$ \pm (\gamma_{0_+} - \gamma_{0_-})$
& $\pm 2 m_0$
&  
$-1$ 
\\\hline
\parbox[c][2.5em][c]{0em}{}
Whittaker 
& $\pm \gamma_{\infty_+}$
& $\pm m_\infty$
&  $2$
\\\hline
\parbox[c][2.5em][c]{0em}{}
Weber 
& $\pm \gamma_{\infty_+}$
&$\pm m_\infty$
&  $1$
\\\hline
\end{tabular}
\end{center}
\caption{ The (unrefined) BPS structures (active classes with corresponding central charges and BPS invariants) associated 
to the quadratic differentials of hypergeometric type corresponding to Figure \ref{fig:confdiag} (the Airy and degenerate Bessel cases are trivial). 
For $p$ an even order pole of $\varphi$,
$\gamma_{p_{\pm}}$ denotes the positively oriented small circle around $p_{\pm}$ (thanks to automorphisms of $\mathbf{P}^1$, singularities are chosen to always lie at $p=0,1$, or $\infty$). To obtain the corresponding refined BPS structures, the column with $\Omega$ should be replaced by their refined versions in Definition \ref{def:newRBPS} below. The quantum correction to the central charge is obtained by replacing $m$ with $\frac{\mu \mathscr{Q}}{2}$ in $Z(\gamma_{\rm BPS})$.
}
\label{table:BPS-str-HG}
\end{table}

\subsection{Definition of $\Omega$}
Having given the construction corresponding to the unrefined case, we can now define the refined BPS structures which are central to this paper. In fact, the lattice and central charge remain the same, with the difference being that $\Omega$ is no longer concentrated purely at $n=0$. To be precise,
\begin{definition}\label{def:newRBPS}
For each $\gamma\in\Gamma$, we define {$ \Omega(\gamma)$} as follows\footnote{There is a fifth possibility, saddles bounding a so-called ``nondegenerate ring domain'' which plays a role in the general case. However, none of the curves considered in this paper contain such a degeneration, and thus we do not discuss them.}: 
\begin{equation} \label{eq:newbps}
\Omega(\gamma)= \begin{cases} 
      1 & \quad \text{if $\gamma$ is a BPS cycle associated with a type I saddle,} \\
      t+t^{-1} & \quad \text{if $\gamma$ is a BPS cycle associated with a type II saddle,} \\
      t^2+2+t^{-2} & \quad \text{if $\gamma$ is a BPS cycle associated with a type III saddle}, \\
      -t^{-1} & \quad \text{if $\gamma$ is a BPS cycle associated with a degenerate ring domain}
   \end{cases}
\end{equation}
and we set $\Omega(\gamma) = 0$ for any non-BPS cycles $\gamma \in \Gamma$.
\end{definition}
\noindent It is clear that setting $t=1$ we recover the (integer-valued, unrefined) BPS invariants of Table \ref{table:BPS-str-HG} \cite{IWAKI2022108191}.

\begin{remark} One may argue that this definition is somewhat arbitrary. At the very least, it is a correct one that will make our formula hold, though it is not the unique such choice. In Proposition \ref{prop:Blemma} and below, we characterize and discuss its uniqueness -- in essence, $\Omega_n+\Omega_{-n}$ must remain fixed, but $\Omega_{n}$ and $\Omega_{-n}$ can be varied without spoiling our main theorem and conjecture. Requiring $\Omega_n$ to be integer-valued obviously limits this, but an infinite amibiguity remains. Further investigation is needed to justify fully the definition above (for example, to see how the approach of \cite{iwaki2023topological} generalizes to the quantum Riemann-Hilbert problem \cite{BBS}). In a slightly different context, some of this jumping behaviour is foreshadowed in \cite{iwaki2016exact, 2000297}, where a similar expression appears in the description of the connection formula for WKB solutions across a type II saddle. Finally, we note that one significant motivation for the choice of value for the degenerate ring domain is that, although it appears unmotivated here, it will be obtained as a genuine Donaldson-Thomas invariant in \cite{kidwai:toappear}.
\end{remark}

\begin{remark}We note furthermore that the definition makes perfect sense for any (reasonable) quadratic differential (whenever $\gamma$ arises from a saddle of type I, II, III, or IV), not only the examples considered in this paper. It would be very interesting to investigate the case where the resulting refined BPS structure is coupled and determine whether our $\Omega$ indeed satisfies the Kontsevich-Soibelman wall-crossing formula. Although we did not treat the case of nondegenerate ring domains here, one reasonable candidate, as in the unrefined case, would be to propose that type I and type V saddles remain unchanged from their usual value in Gaiotto-Moore-Neitzke/Bridgeland-Smith's theory, i.e we assign $\Omega(\gamma)=-(t+t^{-1})$. This provides a definition of a BPS structure associated to any generic quadratic differential, without any constraint on the genus of the spectral curve.
\end{remark}

\begin{remark}The values of $\Omega(\gamma)$ defined here for type II, III, and IV saddles have not appeared before. Although the numbers $2$ and $4$ in the unrefined case (setting $t=1$) agree with the BPS invariants obtained by Haiden \cite{Haiden:2021qts} in a slightly different context, it seems when passing to the refined case we do \textbf{not} obtain the same invariants \cite{kidwai:toappear}. More precisely, in \cite{Haiden:2021qts} a certain 3-Calabi-Yau triangulated category was constructed whose stability conditions correspond to quadratic differentials with simple zeroes and simple poles only. One might hope that a similar construction could be made to cover the case in which higher order poles occur as well. However, the refined DT invariants associated to type II and III saddles in that work were precisely $2$ and $4$ (with no $|n|>0$ terms). Thus, while seemingly related, our $\Omega$ (or their shifts preserving $\Omega_n+\Omega_{-n}$) are not the BPS invariants of the ${\rm CY}3$ triangulated category discussed in that work. It is natural to ask whether they arise as refined DT invariants of some other ${\rm CY}3$ triangulated category. However, this seems difficult to obtain since the value for the type II saddle does not resemble the ``virtual motive'' of any obvious space, which it would need to for such an interpretation.
\end{remark}

\begin{remark}\label{rem:physics}
Indeed, in the physical picture \cite{GAIOTTO2013239,Dimofte:2009tm,dimofte2010refined}, we have that the ``refined BPS index'' is given by a trace over some Hilbert space\footnote{A ``univeral hypermultiplet'' has been factored out in this formula -- see \cite{GAIOTTO2013239,Dimofte:2009tm,dimofte2010refined}.} $\mathcal{H}_{\rm BPS,\gamma}$ of BPS particles of fixed charge $\gamma$:
\begin{equation}
\Omega^{\rm ref}(\gamma,y)={\mathrm{Tr}}_{\mathcal{H}_{\mathrm{BPS},\gamma}} (-y)^{2J_3}
\end{equation}
Here, the $J_3$ appearing in the exponent is a symmetry generator (of spatial rotations around an axis), and $y$ is a formal variable (the refinement parameter) keeping track of the spin of particles. However, in all of the cases in which a double pole is present, there is an ${\rm SU}(2)$ flavour symmetry associated to that point. Since type II saddles can be thought of as the collision of a second order pole with one of the zeroes of a type I saddle, it is then natural to seek an answer as some extension of $\Omega(\gamma)$ in which in additional to the variable $y$ we introduce ``flavour'' fugacities $t_i$ with symmetry generators $F_i$ as exponents, and consider some appropriate specialization/identification of the parameters $t_i$. We do not know if such an index would agree with the values observed here.
\end{remark}

\subsection{Quantum central charge}\label{sec:quantumZ}
Finally, we can enhance the refined BPS structure with the remaining data needed to write down our formula for $F_g$.

Indeed, since our BPS structures are not arbitrary, but rather come from a (refined) spectral curve, they come equipped with one-forms $\omega_{g,1}$, almost all of which are residue-free away from $\mathcal{R}$, and whose integral along cycles descends to homology. The only exception is $\omega_{\frac12,1}$, which can have residues at ramification points. However, one observes that this arises entirely from the $-\frac{\mathscr{Q}}{2}\frac{dy}{y}$ term appearing in the definition, is eliminated by taking the odd part:
\begin{equation}
    \omega_{\frac12,1}^{\rm odd} := \frac{\mathscr{Q}}{2}\sum_{p\in\mathcal{P}_{\!\mathsmaller{+}}'}\mu_{p} \eta_p\label{incorrect formula of 1/2,1}
\end{equation} and define the \emph {quantum central charge}:
\begin{equation}
    \widehat{Z}(\gamma)=Z+ \hbar\, \zbar :=\int_\gamma \omega_{0,1}+\hbar\!\int_{\gamma} \omega_{\frac12,1}^{\rm odd}.
\end{equation}
of $\mathcal{S}^{\bm \mu}$, where $\zbar$ is a (canonical thanks to the spectral curve) {quantum correction to the central charge}. 
We do not understand its meaning in terms of BPS or stability structures, yet our formula shows it is the natural quantity appearing in the free energy. We hope it can be understood in terms of an appropriate generalization of the theory.
\begin{remark}
    One may ask why we did not consider higher-order $\hbar$ corrections in $\widehat{Z}$. This is simply because
    \begin{equation}
    \sum_{g\geq1}^{\infty}\hbar^{2g}\int_{\gamma}\omega_{g,1}=0
    \end{equation}
    due to the residue-free property, so it has no effect.
\end{remark}

\section{Explicit formula for Voros coefficients and free energies}\label{sec:main}

We now turn to the main computation of the paper. First, we need to study some integrals.

\subsection{Limits of integrals}\label{sec:limits}
In contrast with the unrefined case, it is not guaranteed that the terms
\begin{equation}
\int^{p}_{q_0}\cdots\int^{p}_{q_{n}}\omega_{g,n+1}
\end{equation}
\noindent  appearing in the wavefunction \eqref{eq:wavefn} are continuous in $p$. This is due to the pole structures of $\omega_{g,n}$, which are no longer confined to have poles at ramification points alone -- although continuity holds at generic points $p$, one must be careful as $p$ approaches points in $\mathcal{P}$. Fortunately, we can show that in the Weber and Whittaker curves, continuity does in fact hold (it is easy to check that this fails in examples when $\varphi$ has a second order pole).

\begin{lemma}[Main Lemma]\label{lem:continuous}
    Suppose $\mathcal{S}^{\rm \mu}$ is the Weber or Whittaker refined spectral curve, and choose $q_i\in\mathcal{P}$, $i=1\ldots n+1$. For all $2g+n\geq2$, the function
    
    \begin{equation}
          I_{g,n+1}(p)=\int^{p}_{q_0}\cdots\int^{p}_{q_{n}}\omega_{g,n+1}
    \end{equation}
    \noindent is continuous along the path $\alpha$, in particular at the endpoint $p=\sigma{(q_i)}$,    \begin{equation}
        I_{g,n+1}=F_{g,n+1}
    \end{equation}
    as functions of $p$ in a neighbourhood (along $\alpha$) of $\sigma(q_i)$.
\end{lemma}

\begin{proof}
We treat the case with all $q_0,\ldots,q_n$ equal to lighten the notation (the other cases follow by a partial application of the same reasoning, at which point the next integral is from $\sigma(q)$ to itself). For convenience, we use a global coordinate $z$ on $\overline{\Sigma}$, and denote its dummy variables as $\zeta$, $\zeta_i$, etc. Furthermore, we use the notational convention that $\omega(\zeta)$ etc. for $\omega$ a function, section, etc. denotes the coordinate representation (local function) in a given coordinate $\zeta$ etc.

We rewrite the integral with varying endpoint $p$ instead as an integral depending on a parameter $p$ but with fixed endpoints
\begin{equation}\label{eq:covint}
    I_{g,n+1}(p)=\int_{q}^{\sigma({q})}\!\!\cdots \int_{q}^{\sigma({q})}h^*_p\,\omega_{g,n+1}
\end{equation}
where $h_{p}:\overline{\Sigma}\rightarrow \overline{\Sigma}$ is a ($p$-dependent) diffeomorphism\footnote{For example, if we take a coordinate $z$ with $\mathcal{R}=\{0,\infty\}$, $\mathcal{P}=\{\pm 1\}$, we can use $\zeta_i=-1+\frac{z+1}{2}(\xi_i+1)$.} taking $q,\sigma(q)$ to $q,p$ and we abuse notation using the same symbol for the product diffeomorphism (we denote the resulting dummy coordinates by $\xi=h^{-1}_p\circ\zeta$).
We will need the intermediate integrals ($0\leq k < n+1$): \begin{equation}
I_{{g},{n}+1}^{({k})}(\zeta_0,\ldots,\zeta_{n-k};p):=\int_{q}^{\sigma(q)}\!\!\cdots\int_{q}^{\sigma(q)} h^{*}_p\omega_{{g},{n}+1}(\zeta_0,\ldots \zeta_{n-k},\xi_{n-k+1},\ldots,\xi_n)\end{equation}
where the pullback acts on the last $k$ factors. It was shown Lemma A.5 of \cite{kidwai2023quantum} that (a fortiori) $I^{(k)}_{g,n+1}$ exists whenever $\zeta_i\neq \sigma(\zeta_j)$ and $\zeta_i\neq \sigma(p)$, is independent of the path between the endpoints of the integral and has vanishing residues at all points and is \emph{regular} at $\zeta_i = p, q,\sigma(q)$. Thus all our concerns will be concentrated on the poles at $\sigma(p)$ (in a neighbourhood of $q$), and we can chop off the rest of the path without loss of generality.

Since we are free to choose the path of integration (as long as we avoid any singularities on the interior), we can always choose $\tilde{\gamma}:[0,1]\rightarrow\overline{\Sigma}$ (ending on $p$) to ensure that for $p$ approaching $\sigma(q)$ along $\alpha$, we have \begin{equation}
\label{eq:ineq}|z(p)-z(\sigma(q))|<|\zeta_{n-k}-z(\sigma(p))|,
\end{equation}
\noindent for $\zeta_{n-k}$ lying in $z(\tilde{\gamma})$

\begin{figure}[h]
    \centering
    \includegraphics[width=0.35\linewidth]{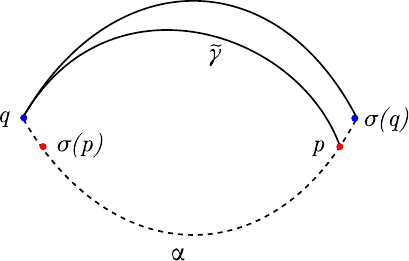}
    \caption{Paths $\tilde{\gamma}$ and $\alpha$. The unlabelled path is obtained by applying $h_p^{-1}$ to $\tilde{\gamma}$.}
    \label{fig:paths}
\end{figure}

We will show that the absolute value of $I^{(k)}_{g,n+1}(\xi_0,\ldots,\xi_{n-k};p)$ is uniformly bounded in the parameter $p$ and variable $\xi_{n-k}$ along the path $h_p(\gamma)$ for all $k<n+1$. Then by the dominated convergence theorem,

 \begin{align}
\lim_{p\rightarrow\sigma(q)}I_{g,n+1}(p)&=\int_{q}^{\sigma(q)}\lim_{p\rightarrow\sigma(q)} h^{*}_pI^{(n)}_{g,n+1}(\xi_0;p)\\ &\hspace{2.1cm}\vdots\\   &=\int_{q}^{\sigma(q)}\!\!\cdots \int_{q}^{\sigma(q)}\lim_{p\rightarrow \sigma(q)}h^*_p I_{g,n+1}^{(k)}(\xi_0,\ldots, \xi_{n-k};p)\\ &\hspace{2.1cm}\vdots\\
&=\int_{q}^{\sigma(q)}\!\!\cdots \int_{q}^{\sigma(q)}\lim_{p\rightarrow\sigma(q)}h^{*}_p\omega_{g,n+1}(\xi_0,\ldots \xi_n)
=I_{g,n+1}(\sigma(q))
\end{align}
as desired.
 \normalsize
We note that due to the dependence of the integrands on $p$ which comes from pulling back by $h_{p}$ (which has a simple derivative causing no difficulties), it is sufficient to show that for any $g,n$ and $k$, $I^{(k)}_{g,n+1}(\zeta_0,\ldots,\zeta_{n-k};p)$ is uniformly bounded along the path $\gamma$ (in the $\zeta_{n-k}$ variable) in the parameter $p$ (valued along $\alpha$), automatically ensuring uniformity in $\xi_{n-k}$. 

Finally, to show uniform boundedness in $\zeta$, note due to all functions being rational it is sufficient to show that for all $r\geq0$ and all $\zeta$ along $\gamma$ (and near $q$), the coefficient of $(\zeta_{n-k}-\zeta)^{-r}$ in the Laurent expansion of $f$ at $\zeta$ (viewed as a function of $z$) has a zero of order at least $r$ at $z=\sigma(q)$ (the converse also holds), so that \eqref{eq:ineq} can be applied.


We proceed by induction on $2g+n$. For $I_{0,3}^{(k)}$, it follows from the usual topological recursion, and $I_{1,1}$ because it is a single integral. For $I_{\frac12,2}^{(1)}$, the result follows by explicit calculation and inspection of the poles and zeroes, together with \eqref{eq:ineq}. The crucial point is that the coefficient of the pole (order $-2$ only) $I^{(1)}_{\frac12,2}(\zeta_0)$ at $\zeta_{0}=z(\sigma(q))$ has a double zero in $z$ (indeed this does not hold for e.g. the Bessel curve, as there is only a simple zero in $z$, so that \eqref{eq:ineq} does not suffice).


As the inductive assumption, we assume that for all $(\tilde{g},\tilde{n})$ with $1<m:=2\tilde{g}+\tilde{n}<2g+n $ $(=m+1)$, and all $\tilde{k}<\tilde{n}+1$, $I_{\tilde{g},\tilde{n}+1}^{(\tilde{k})}$is bounded uniformly along the (truncated) $\tilde{\gamma}$ in $\zeta_{n-k}$ for $p$ lying on $\alpha$ converging to $\sigma(q)$.




Integrating (we always integrate starting with $\zeta_n$ and proceed leftward to $\zeta_0$) the global loop equation \eqref{pole2} and using the refined recursion \eqref{eq:RTR}, we obtain the following  ``$\Delta\omega_{0,2}$ terms'' in $I^{(k)}_{g,n+1}$:
\begin{align}\label{eq:deltaterms} &-\int_{q}^{p}\!\!\ldots\int_{q}^{p}\sum_{j=1}^{n}\dfrac{\Delta\omega_{0,2}(\zeta_0,\zeta_j) \cdot \omega_{g,n}(\zeta_0,\ldots,\widehat{\zeta}_{j},\ldots,\zeta_{n})}{2y(\zeta_0)x'(\zeta_0)}\nonumber\\&\hspace{0mm}=-\sum_{j=1}^{n-k}\dfrac{\Delta\omega_{0,2}(\zeta_0,\zeta_j) \cdot I_{g,n}^{(k)}(\zeta_0,\ldots,\widehat{\zeta}_{j},\ldots,\zeta_{n-k};p)}{2y(\zeta_0)x'(\zeta_0)}\,\,-\!\!\!\sum_{j=n-k+1}^{n}\!\!\left(\frac{\int^{p}_{q}\Delta\omega_{0,2}(\zeta_0,\cdot\,)}{2y(\zeta_0)x'(\zeta_0)}\right)I_{g,n}^{(k-1)}(\zeta_0,\ldots{{\zeta}_{n-k}};p) 
\end{align}
\normalsize

\noindent where $\Delta\omega(p_0,p_1):=\omega(p_0,p_1)-\omega(\sigma(p_0),p_1)$. Since $I^{(k)}_{g,n}$ was uniformly bounded in the variable $\zeta_{n-k}$ by a constant $C^{(k)}_{g,n}(\zeta_0,\ldots,\widehat{\zeta_j},\ldots,\zeta_{n-k-1})$ (independent of $p$) along the path of integration, we have that the $i$th term in the first sum is bounded by \begin{equation}
    C_{g,n}^{(k)}(\zeta_0,\ldots,\widehat{\zeta_j},\ldots,\zeta_{n-k-1}) \sup_{\zeta_{n-k}\in\tilde{\gamma}}\Bigg| \dfrac{\Delta\omega_{0,2}(\zeta_0,\zeta_j)}{2y(\zeta_0)x'(\zeta_0)}\Bigg|.
\end{equation}
\noindent which is a constant independent of $p$ (if $n-k=0$, the sum does not appear at all).

For the second term, we note that $\int_{q}^{p}\Delta \omega_{0,2}(\zeta_0,\zeta_i)=\eta_{p}(\zeta_0)-\eta_{q}(\zeta_0)$ has poles (in $\zeta_0$) only at $p,\sigma(p),q,\sigma(q)$, and they are simple. Since $ydx$ has at least a second order pole at $\mathcal{P}$, the term inside the brackets has a pole at $\zeta_0=p,\sigma(p)$ only\footnote{Note zeroes of $ydx$ must lie in $\mathcal{R}$ by the definition of refined spectral curve -- Assumption 2.7 of \cite{kidwai2023quantum}, so we certainly do not have any poles along the path from the reciprocal.}, so the lower endpoint does not contribute at all. For the upper endpoint, it is easy to check for the term in brackets that the pole in $\zeta_0$ at $\sigma(p)$ is accompanied by a zero in $z$ at $\sigma(q)$, so it is uniformly bounded thanks to \eqref{eq:ineq} as long as we consider the path only near $q$, and the inductive assumption ensures the product is too.


Next we look at the ``generic'' (product) terms. Letting $J_{i}'$ denote the subsets of indices which are not being integrated out, we have:

\begin{align}
    -\int_{q}^{p}\ldots \int_{q}^{p}\frac{\omega_{g_1,n_1+1}(\zeta_0, \zeta_{J_1})\,\omega_{g_2,n_2+1}(\zeta_0,\zeta_{J_2})}{2y(\zeta_0)x'(\zeta_0)}  &=-\dfrac{I^{(k_1)}_{g_1,n_1+1}(\zeta_0,\zeta_{J_{1}'};p)\,I^{(k_2)}_{g_2,n_2+1}(\zeta_0,\zeta_{J_{2}'};p)}{2y(\zeta_0)x'(\zeta_0)}
\end{align}


Unless some $(g_1,n_1)$ or $(g_2,n_2)$ equals $(\frac12,1)$, uniform boundedness follows from the inductive assumption (in the $n-k=0$ case, since the denominator has poles near $p=\sigma(q)$). For terms of the form:

    \begin{equation}
        -\int_{q}^{p}\cdots\int_{q}^{p}\frac{\omega_{\frac12,1}(\zeta_0)\omega_{g-\frac12,n+1}(\zeta_0,\ldots,\zeta_n)}{2y(\zeta_0)x'(\zeta_0)}=-\left(\frac{\omega_{\frac12,1}(\zeta_0)}{2y(\zeta_0)x'(\zeta_0)}\right)I^{(k)}_{g-\frac12,n+1}(\zeta_0,\ldots,\zeta_{n-k};p).
    \end{equation}
we know $\omega_{\frac12,1}$ is holomorphic except at $\mathcal{R}$ and $\mathcal{P}$, but that its poles at $\mathcal{P}$ are simple, so thanks to the poles of $ydx$ at $\mathcal{P}$, the term in the brackets is uniformly bounded along the path. By the inductive assumption, we have uniform boundedness again.

For the third term, we have that $I^{(k)}_{g-1,n+2}({\zeta_0,\zeta_1;p})$ has potential poles in $\zeta_i$ only at $\zeta_0=\sigma(\zeta_1)$ and $\sigma(p_j)$. Thus
    \begin{equation}
-\int_{q}^{p}\cdots\int_{q}^{p}\dfrac{\omega_{g-1,n+2}(\zeta_0,\zeta_0,\zeta_1,\ldots,\zeta_{n})}{2y(\zeta_0)x'(\zeta_0)}=\dfrac{I^{(k)}_{g-1,n+2}(\zeta_0,\zeta_0,\zeta_1,\ldots,\zeta_{n-k};p)}{{2y(\zeta_0)x'(\zeta_0)}}
    \end{equation}
    can only have poles at $\zeta_0\in\mathcal{R}$. Again, by the assumption, the result is bounded uniformly.

For the fourth term,

    \begin{equation}
        -\mathscr{Q}\,\int_{q}^{p}\cdots\int_{q}^{p}\dfrac{x'(\zeta_0)\cdot \partial_{\zeta_0}\left(\frac{ \omega_{g-\frac12,n+1}(\zeta_0,\ldots,\zeta_{n})}{dx(\zeta_0)}\right)}{y(\zeta_0)x'(\zeta_0)}=-\mathscr{Q}\,\dfrac{1}{2y(\zeta_0)}\cdot \partial_{\zeta_0}\left(\frac{ I_{g-\frac12,n+1}^{(k)}(\zeta_0,\ldots,\zeta_{n-k};p)}{x'(\zeta_0)}\right)
    \end{equation}
By the product rule, we obtain the product of a uniformly bounded $I^{(k)}_{g-\frac12,n+1}$ with $\frac{-\partial_{\zeta_0}y}{2y^2x'}= \frac{\partial_{\zeta_0}\log y}{y(\zeta_0)x'(\zeta_0)}$ which is independent of $p$ (the total derivative term doesn't cause any issue since although $\partial_{\zeta_0}$ increase the order of poles in $\zeta_0$ by $1$, the $y(\zeta_0)x'(\zeta_0)$ in the denominator ensures the zero (in $p$) in the coefficient increases too, so that inequality \eqref{eq:ineq} is sufficient).

Finally, we have the derivative term in \eqref{pole2}:
\begin{align}
    &\sum_{j=1}^n\int_{q}^{p}\ldots \int_{q}^{p}\partial_{\zeta_j}\left(\frac{\eta_{\zeta_j}(\zeta_0)}{2y(\zeta_j)x'(\zeta_j)}\cdot \omega_{g,n}(\zeta_1,\ldots,\zeta_n)\right)\\&=\sum_{j=1}^{n-k}\partial_{\zeta_j}\left(\frac{\eta_{\zeta_j}(\zeta_0)}{2y(\zeta_j)x'(\zeta_j)}\cdot I^{(k)}_{g,n}(\zeta_1,{...},\zeta_{n-k};p)\right)+\!\!\!\!\!\sum_{j=n-k+1}^{n}\int_{q}^{p}\partial_{\zeta_j}\left(\frac{\eta_{\zeta_j}(\zeta_0)}{2y(\zeta_j)x'(\zeta_j)}\cdot I^{(k-1)}_{g,n}(\zeta_1,{...},\zeta_{n-k},\zeta_j;p)\right) \\
    &=\sum_{j=1}^{n-k}\partial_{\zeta_j}\left(\frac{\eta_{\zeta_j}(\zeta_0)}{2y(\zeta_j)x'(\zeta_j)}\cdot I^{({k})}_{g,n}(\zeta_1,\ldots,\zeta_{n-k};p)\right)+\!\!\!\!\!\sum_{j=n-k+1}^n\frac{\eta_{\zeta_j}(\zeta_0)}{2y(\zeta_j)x'(\zeta_j)}\cdot I^{({k-1})}_{g,n}(\zeta_1,\ldots,\zeta_{n-k},\zeta_j;p)
    {\Bigg|}_{\zeta_j=z(q)}^{\zeta_j=z(p)}
\end{align}
where the derivative term causes no issue as above (since the inside is uniformly bounded by the inductive assumption and inspection of $\frac{\eta_p(\zeta_0)}{2y(\zeta_0)x'(\zeta_0)}$ with \eqref{eq:ineq}). The contribution from the lower endpoint vanishes since the denominator contains $y(\zeta_j)x'(\zeta_j)$. For the upper endpoint, if $n\neq k$ there is no problem by the inductive assumption. If $n=k$ the only dependence on $\zeta_0$ appears in the brackets, whose pole structure can be inspected and is uniformly bounded by \eqref{eq:ineq}.


Thus, we conclude that for all $(g,n)$ with $2g+n=m+1$ and $k<n+1$, all $I^{(k)}_{g,n+1}$ are uniformly bounded in the variable $\zeta_{n-k}$ along (truncated) $\tilde{\gamma}$ as $p$ tends to $\sigma(q)$ along $\alpha$, completing the induction.

\end{proof}

\subsection{Difference relations}
The variational formula \eqref{Fvariation} can be used to obtain a fundamental relationship between Voros coefficients and the free energy. More precisely, we obtain the following difference relation between $V$ and $F$, generalizing results of \cite{iwaki2023voros,iwaki2019voros} to the refined setting:

\begin{proposition}\label{prop:diffreln}
Suppose $\mathcal{S}^{{\bm \mu}}$ is either the Weber or Whittaker curve, and let $\alpha$ denote a path from $\infty_-$ to $\infty_+$ avoiding ramification points. We have:


\begin{equation}\label{eq:diffreln}
V_{\alpha}=F\left(m-\frac{\nu-1}{2\beta^\frac12}\hbar\right)    - F\left(m-\frac{\nu+1}{2\beta^\frac12}\hbar\right) -\frac{1}{{\beta^{\frac12}}\hbar}\dfrac{\partial F_0}{\partial m}+\frac{\nu}{2\beta} \dfrac{\partial^2 F_0}{\partial^2 m}-\frac{1}{{\beta^{\frac12}}}\dfrac{\partial F_{\frac12}}{\partial m}
\end{equation}
 understood as a relation between formal series in $\hbar$.
\end{proposition}
\begin{proof}
The proof is similar to \cite{iwaki2019voros}, and follows easily thanks to the continuity property shown in Lemma \ref{lem:continuous}. From the definition of Voros coefficient,

\begin{align}
V_{\alpha}
&=\sum_{k=1}^{\infty}\hbar^k\int_{\infty_-}^{\infty_+}\left(\sum_{\substack{2g-2+n=k \\ g\geq0,n\geq1}} \dfrac{1}{\beta^{n/2}n!}\dfrac{d}{dz}\int_{D(z)}\!\!\!\!\!\ldots\int_{D(z)}  \omega_{g,n}(\zeta_1,\ldots,\zeta_n)\right)
\\
&=\sum_{k=1}^{\infty}\hbar^k\sum_{\substack{2g-2+n=k \\ g\geq0,n\geq1}} \dfrac{1}{\beta^{n/2}n!}\left(\int_{D(\infty_+)}\!\!\!\!\!\ldots\,\,\int_{D(\infty_+)}-\,\,\,\int_{D(\infty_-)}\!\!\!\!\!\ldots\,\,\int_{D(\infty_-)}\right)  \omega_{g,n}(\zeta_1,\ldots,\zeta_n)
\end{align}
where we used Lemma \ref{lem:continuous} to write the integral in terms of $D(\infty_{\pm};{\bm \nu})$. Then since 
\begin{equation}
    D(\infty_+)=\nu_{\infty_-}([\infty_+]-[\infty_-]), \qquad D(\infty_-)=-\nu_{\infty_+}([\infty_+]-[\infty_-])
\end{equation}

\noindent we get 

\begin{align}
    V_\alpha(\hbar)&=\sum_{k=1}^\infty \hbar^k \sum_{\substack{2g+n-2=k\\g\geq0,n\geq 1}}\dfrac{\nu_-^n-(-\nu_+)^n}{n!\beta^{n/2}}\int_{\infty_-}^{\infty_+}\ldots\int_{\infty_-}^{\infty_+}  \omega_{g,n}(\zeta_1,\ldots,\zeta_n)\\
    &=\sum_{n=1}^\infty \dfrac{\nu_{-}^n-(-\nu_{+})^n}{n!\beta^{n/2}}\hbar^n \dfrac{\partial^n F (m;\hbar)}{\partial m^n} - \dfrac{\nu_{-}-(-\nu_{+})}{\beta^{\frac12}\hbar}\dfrac{\partial F_0}{\partial m}-\frac{\nu_-^2-(-\nu_+)^2}{2\beta} \dfrac{\partial^2 F_0}{\partial^2 m}-\frac{\nu_{-}-(-\nu_{+})}{{\beta^{\frac12}}}\dfrac{\partial F_{\frac12}}{\partial m}
\end{align}
where we used the variational formula \eqref{Fvariation} in the second line.
Since we started at $m=1$, we subtract off the lowest order terms in $F$ as well, writing $\nu:=\nu_{+}-\nu_-$ to obtain the result.

\end{proof}

\begin{remark}
    This relation breaks down in the other examples when considering Voros coefficients associated to paths between the two preimages of a second order pole of $\varphi$, due to the failure of Lemma \ref{lem:continuous}. We will return to the issue of second order poles in future work.
\end{remark}

In order to use the relation \eqref{eq:diffreln} to learn about $F$, we must first know something about $V$. Fortunately, various ``contiguity relations'' between Voros coefficients exist, which we can exploit.
\begin{lemma}\label{lemm:Vrelns} The Voros coefficients $V=V_\alpha$ satisfy:
\begin{itemize}
    \item For the Weber curve, 
\begin{align}
    V(\nu+2\beta)-V(\nu)=&-\log\left(m+\hbar\frac{\mathscr{Q} \mu}{2}+\frac{\epsilon_2}{2}\nu-\frac{\epsilon_1}
{2}\right)+\log m.
\end{align}
In particular, when $\nu=-\beta$,
\begin{equation}
     V(\beta)-V(-\beta)=-\log\left(m+\hbar\frac{\mathscr{Q} \mu}{2}\right)+\log m\label{VVWeb}
\end{equation}
\item
For the Whittaker curve, 
\begin{align}
    V(\nu-2\beta)-V(\nu)=&\log\left(m+\hbar\frac{\mathscr{Q} \mu}{2}+\frac{\epsilon_2}{2}\nu+\frac{\epsilon_1}
{2}+\frac{\epsilon_1+\epsilon_2}{2}\right)\nonumber\\
&\hspace{0mm}+\log\left(m+\hbar\frac{\mathscr{Q} \mu}{2}+\frac{\epsilon_2}{2}\nu+\frac{\epsilon_1}
{2}-\frac{\epsilon_1+\epsilon_2}{2}\right) -2\log m
\end{align}
where
In particular, when $\nu=\beta$,
\begin{align}
    V(\beta)-V(-\beta)=&-\log\left(m+\hbar\frac{\mathscr{Q} \mu}{2}+\frac{\hbar\mathscr{Q}}{2}\right) -\log\left(m+\hbar\frac{\mathscr{Q} \mu}{2}-\frac{\hbar\mathscr{Q}}{2}\right) +2\log m.\label{VVWhi}
\end{align}
\end{itemize}
\end{lemma}
\begin{proof}
This follows from a similar argument used in \cite{iwaki2023voros,publishedpapers/8600728} which treat a special case.


 \end{proof}

Combining these with the difference relation between $V$ and $F$, we can give a difference equation for the free energy itself. In order to simplify notation, let us first define an operator $D$ as:
\begin{equation}
     \Delta_{\epsilon_1,\epsilon_2}:=-\left(e^{\frac{1}{2}\epsilon_1\partial_m}-e^{-\frac{1}{2}\epsilon_1\partial_m}\right)\left(e^{\frac{1}{2}\epsilon_2\partial_m}-e^{-\frac{1}{2}\epsilon_2\partial_m}\right)
\end{equation}
Then, $\Delta$ acts on functions $f(m;\hbar)$ of $m$ and $\hbar$, as
\begin{align}
   &\Delta_{\epsilon_1,\epsilon_2}\cdot f(m;\hbar)\nonumber\\
   &= -f\left(m-\frac{\epsilon_1}{2}-\frac{\epsilon_2}{2};\hbar\right)+  f\left(m-\frac{\epsilon_1}{2}+\frac{\epsilon_2}{2};\hbar\right)+f\left(m+\frac{\epsilon_1}{2}-\frac{\epsilon_2}{2};\hbar\right)-f\left(m+\frac{\epsilon_1}{2}+\frac{\epsilon_2}{2};\hbar\right).
\end{align}

\begin{proposition}\label{prop:Fdiff}
The free energy for the Weber curve satisfies the difference equation  \medskip
\begin{align}
\Delta_{\epsilon_1,\epsilon_2}\cdot F=\log\left(m+\frac{\mu}{2} \mathscr{Q}\hbar\right) \label{F4Web}
\end{align}
\medskip
and the free energy for the Whittaker curve satisfies\medskip
\begin{align}
   \Delta_{\epsilon_1,\epsilon_2}\cdot F=\log{\left(m+\frac{\mu}{2} \mathscr{Q}\hbar+\frac{\mathscr{Q}\hbar}{2}\right)}+\log{\left(m+\frac{\mu}{2} \mathscr{Q}\hbar-\frac{\mathscr{Q}\hbar}{2}\right)}. \label{F4Whi}
\end{align}
\end{proposition}
\vspace{1mm}\begin{proof}

This is a straightforward consequence of Proposition \ref{prop:diffreln} and \eqref{VVWeb}, \eqref{VVWhi} in Lemma \ref{lemm:Vrelns}


\end{proof}
\noindent 
As expected, \eqref{F4Web}, \eqref{F4Whi} recover exactly the difference equations for $F$ in \cite{iwaki2019voros} in the unrefined limit.
\subsection{Free energy}
Equipped with the difference equation for $F$, it is now straightforward to solve it and obtain its explicit form, following the same technique as \cite{iwaki2023voros,iwaki2019voros}.

\begin{theorem}\label{thm:main} Let $\mathcal{S}^{\bm \mu}= (\overline{\Sigma},x,y,D({\bm \mu}))$ denote the Weber, Whittaker, degenerate Bessel, or Airy refined spectral curve. For half-integers $g>1$, the genus $g$ refined topological recursion free energy of $\mathcal{S}^{\bm \mu}$ can be expressed in terms of the corresponding refined BPS structure $(\Gamma,Z,\Omega)$ together with the quantum correction $\zbar$ as


   \begin{equation}\label{eq:mainformula}    F_{g}=(-1)^{2g-2}\sum_{\gamma\in\Gamma} \sum_{{ n \in \mathbb{Z}}} \dfrac{{\mathsf{B}_{2,2g}^{[n]}(\gamma)}}{4g(2g-1)(2g-2)}\Omega_n(\gamma)\left( 
{\dfrac{2 \pi i}{Z(\gamma)}} \right)^{2g-2}
\end{equation}
where we write
\begin{equation}
\mathsf{B}_{2,2g}^{[n]}(\gamma):=  B_{2,2g}{\left(\frac{\mathscr{Q}}{2}+\frac{\zbar(\gamma)}{2\pi i}+n\frac{\mathscr{Q}}{2}\,\Big|\,\beta^{\frac{1}{2}},-\beta^{-\frac{1}{2}}\right)}\end{equation}
with $B_{2,2g}$ the double Bernoulli polynomial of degree $2g$.
\end{theorem}
\begin{proof}
Let us define two operators $R_{\text{Web}},R_{\text{Whi}}$ by
\begin{align}
   R_{\text{Web}}:&=e^{\frac12\mu\mathscr{Q}\hbar\partial_m},\\
    R_{\text{Whi}}:&=e^{\frac12(\mu+1)\mathscr{Q}\hbar\partial_m}+e^{\frac12(\mu-1)\mathscr{Q}\hbar\partial_m}.
\end{align}
Then, the difference equations \eqref{F4Web} and \eqref{F4Whi} are respectively written as
\begin{align}
    \Delta_{\epsilon_1,\epsilon_2}\cdot F = R_\mathsmaller{\bullet} \cdot \log m
\end{align}
where $\mathsmaller{\bullet}\in\{\text{Web},\text{Whi}\}$.

By the definition of double Bernoulli polynomials (Appendix \ref{sec:Bernoulli}), for any ${b}\in\mathbb{C}$ we have
\begin{align}
    \Delta_{\epsilon_1,\epsilon_2} \cdot \sum_{k\geq0}-B_{2,k}\left(\frac{{b} \mathscr{Q}}{2}+\frac{\mathscr{Q}}{2}\,|\,\beta^{\frac12},-\beta^{-\frac12}\right)\dfrac{(\hbar\partial_{m})^{k-2}}{k!}=e^{\frac12{b}\mathscr{Q}\hbar\partial_m}\label{B2importantequation}
\end{align}
where $(\partial_m)^{-1}$ denotes the antiderivative which comes with constant ambiguity that we ignore at the moment. Therefore, if we set 
\begin{align}
    \hat F_0&=\frac{1}{2} m^2 \log (m)-\frac{3 m^2}{4},\\
 \hat F_{\frac12}&=\frac12\mu\mathscr{Q} m\log m - \frac12\mu\mathscr{Q}m,\\
 \hat F_1&=-\frac{2+(1-3\mu^2)\mathscr{Q}^2}{24}\log m
\end{align}
and for $g\geq\frac32$,
\begin{align}
    \hat{F}_{g}&=\frac{(-1)^{2g-2}}{2g(2g-1)(2g-2)}B_{2,2g}\left(\frac{\mu\mathscr{Q}}{2}+\frac{\mathscr{Q}}{2}\,|\,\beta^{\frac12},-\beta^{-\frac12}\right)\left(\frac{1}{m}\right)^{2g-2},\label{hatFg}
\end{align}
then \eqref{B2importantequation} implies that $\hat F:=\sum_{g\geq0}\hbar^{2g}\hat F_g$ is a solution to \eqref{F4Web}

Recall from Definition \ref{def:Funstable} that $F_0,F_\frac12,F_1$ have some polynomial ambiguity in $m$, which we can use to set $\hat F_0=F_0,\hat F_\frac12=F_\frac12,\hat F_1=F_1$. Thus, our last task is to show that for $g\geq\frac32$, $\hat F_g=F_g$. Let us define $G:=F-\hat F$, then by construction $\Delta_{\epsilon_1,\epsilon_2}\cdot G=0$. Applying a similar argument to \cite{iwaki2023voros,iwaki2019voros}, we can show that $G_g$ should satisfy $\partial_m^2G_g(m)=0$. Then, by the homogeneity\footnote{This follows for the examples in this paper by a simple argument rescaling $\omega_{0,1}$.} of $F_g$ in $m$ for $g\geq\frac32$, we conclude that $G_g(m)=0$, and we have $F_g=\hat F_g$ for all $g\geq0$. Finally, by using properties of Bernoulli polynomials (Appendix \ref{sec:Bernoulli}), \eqref{hatFg} can be written as the expression of the theorem.

The Whittaker case follows by a similar argument using $R_{\text{Whi}}$ instead. Finally, the degenerate Bessel and Airy cases follow from Appendix \ref{appendix:airydbes}, since their BPS structures are trivial \cite{IWAKI2022108191}.




\end{proof}

\subsection{Voros coefficients}
Finally, we can combine Theorem \ref{theorem:main} with the difference relation to obtain the explicit formula for the Voros coefficients of the refined Weber and Whittaker quantum curves.

It turns out that the cycle Voros coefficients $V_\gamma$ is very easy to describe. We will use the notation $\nu : H_1(\widetilde{\Sigma}, {\mathbb Z}) \to {\mathbb C}$ for the linear map given on generators by $\nu(\pm\gamma) = \pm {\nu}$. Then we have:
\begin{proposition} The cycle Voros coefficient along $\gamma$ is a finite (two-term) series in $\epsilon_1$ written explicitly as:
\begin{equation}\label{eq:cyclecharge}
    V_\gamma(\hbar) = \frac{\widehat{Z}(\gamma)}{\hbar \beta^{\frac12}} - \frac{\pi i \nu(\gamma)}{\beta} 
\end{equation}
\end{proposition}
\begin{proof}
Up to the factors of $\beta$, the $Z$ and $\nu$ term are the same as in the unrefined case and follow from the same argument as in \cite{iwaki2019voros}. Note that $\frac{dy}{y}$ is invariant, and thus does not contribute after taking the odd part; the remaining part $\omega^{\rm odd}_{\frac12,1}$ is anti-invariant, so gives the remaining term. Explicitly,
\begin{equation}
    V_\gamma= \frac{2\pi i m}{\beta^\frac12\hbar}+\frac{\mathscr{Q}\mu}{2\beta^{\frac12}}-\frac{\pi i \nu}{\beta},
\end{equation}
which can be written as \eqref{eq:cyclecharge}.
\end{proof}

The path Voros coefficient is less trivial:
\begin{theorem}
For the Weber and Whittaker curves, let us write the path Voros coefficient along $\alpha$ as
\begin{equation}
V_{\alpha} =  \sum_{k \ge 1} \hbar^{k} V_{\alpha, k}.
\end{equation}
Then $V_{\alpha, k}$ can be written explicitly in terms of the  (single) Bernoulli polynomials $B_{1,k}$ as 
\begin{align}
\label{eq:thm1bmain}
V_{\alpha, k}
& =\frac{(-1)^{k+1}}{2}\sum_{\gamma \in \Gamma }\sum_{n\in\mathbb{Z}} 
\langle{\gamma,\alpha}\rangle\Omega_n(\gamma) \,\frac{\mathsf{B}_{1,k+1}^{[n]}(\gamma)}{k(k+1)}
 \, 
\left( \frac{2 \pi i}{Z(\gamma)} \right)^k.  
\end{align}

\noindent where we write
{\begin{equation}
    \mathsf{B}_{1,k}^{[n]}(\gamma)=B_{1,k}\left(\frac{\beta^{\frac12}
    }{2}-\frac{\nu(\gamma)}{2\beta^{\frac12}}+\frac{{\zbar(\gamma)}}{2\pi i} +n\frac{{\mathscr{Q}}}{2}\,\Big|\,\beta^\frac12\right),
\end{equation}}
and $\langle\cdot,\cdot\rangle$ is the intersection pairing.


\end{theorem}
\begin{proof}
We take a similar approach to \cite{iwaki2023voros,iwaki2019voros}. Let us define $V^{\text{reg}}_{\alpha}$ by
\begin{equation}
    V^{\text{reg}}_{\alpha}:=V_\alpha+\frac{1}{{\beta^{\frac12}}\hbar}\dfrac{\partial F_0}{\partial m}-\frac{\nu}{2\beta} \dfrac{\partial^2 F_0}{\partial^2 m}+\frac{1}{{\beta^{\frac12}}}\dfrac{\partial F_{\frac12}}{\partial m}.
\end{equation}
Then, Proposition \ref{prop:diffreln} implies that $V^{\text{reg}}_{\alpha}$ satisfies
\begin{equation}
    V^{\text{reg}}_{\alpha}=-e^{-\frac12\nu\epsilon_2\partial_m}\left(e^{\frac12\epsilon_2\partial_m}-e^{-\frac12\epsilon_2\partial_m}\right)F.
\end{equation}
Using the difference equation for $F$ (Proposition \ref{prop:Fdiff}), we obtain:
\begin{equation}
    \left(e^{\frac12\epsilon_1\partial_m}-e^{-\frac12\epsilon_1\partial_m}\right)V^{\text{reg}}_{\alpha}=e^{\frac12\nu\epsilon_2\partial_m} R_\mathsmaller{\bullet}\log m.\label{Vdiff2}
\end{equation}
We note that \eqref{Vdiff2} is \emph{not} a repetition of Proposition \ref{prop:diffreln} because \eqref{Vdiff2} describes a shift of $m$ in $V^{\text{reg}}_{\alpha}$, not $\nu$.

We then apply a similar technique as the proof of Theorem \ref{thm:main}. That is, for any ${b}\in\mathbb{C}$, we have
\begin{align}
      \left(e^{\frac12\epsilon_1\partial_m}-e^{-\frac12\epsilon_1\partial_m}\right)\cdot\left(-\sum_{k\geq0}B_{1,k}\left(\frac{{b}+1}{2}\beta^\frac12\,\Big|\,\beta^\frac12\right)\dfrac{(\hbar\partial_{m})^{k-1}}{k!}\right)=e^{\frac12{b}\epsilon_1\partial_m}
\end{align}
so we get that
\begin{equation}{V_\alpha^{\rm reg
}}=\sum_{k\geq1}\frac{{{(-1)^{k+1}}{}}}{{k(k+1)}}B_{1,k+1}\left(\frac{{b}+1}{2}\beta^\frac12\,\Big|\,\beta^\frac12\right)\left(\frac{\hbar}{m}\right)^k\end{equation}solves \eqref{Vdiff2} for $\mathsmaller{\bullet}={\rm Web}$ after setting $
    {b}=-\frac{\nu}{\beta}+\frac{\mu\mathscr{Q}}{\beta^{\frac12}}$. Uniqueness follows after imposing homogeneity as before, and since $\langle\gamma,\alpha\rangle=1$, this can be written as \eqref{eq:thm1bmain}. 

The derivation for the Whittaker curve is completely parallel.


\end{proof}

\subsection{Conjecture for other spectral curves}

Although technical difficulties arise when $\varphi$ has a a second order pole\footnote{Note that we could treat the Voros coefficients associated to a path $\alpha$ between the preimages of a higher-order ($\geq3$) pole without difficulty, even if elsewhere on the surface there are second-order poles. The result would be similar to what we have presented, but to avoid introducing further notation we did not bother to state it here.} (which is the case for all other hypergeometric type curves, cf. Figure \ref{fig:confdiag}), we can still compute free energies in these examples up to some finite order. Based on this, we have found

\begin{conjecture}\label{conj:main} Fix a refined spectral curve of hypergeometric type $\mathcal{S}^{\bm \mu}= (\overline{\Sigma},x,y,D({\bm \mu}))$. The refined topological recursion free energy of $\mathcal{S}^{\bm \mu}$ can be expressed in terms of the corresponding refined BPS structure and quantum correction to the central charge in \S\ref{sec:bpsfromqd} as

    \begin{equation}\label{eq:mainformula}    F_{g}={(-1)^{2g-2}}\sum_{\gamma\in\Gamma} \sum_{{ n \in \mathbb{Z}}} \dfrac{{\mathsf{B}_{2,2g}^{[n]}(\gamma)}}{4g(2g-1)(2g-2)}\Omega_n(\gamma)\left( 
{\dfrac{2 \pi i}{Z(\gamma)}} \right)^{2g-2}
\end{equation}
where we write
\begin{equation}
\mathsf{B}_{2,2g}^{[n]}(\gamma):=  B_{2,2g}{\left(\frac{\mathscr{Q}}{2}+\frac{\zbar(\gamma)}{2\pi i}+n\frac{\mathscr{Q}}{2}\,\Big|\,\beta^{\frac{1}{2}},-\beta^{-\frac{1}{2}}\right)}\end{equation}
with $B_{2,2g}$ the double Bernoulli polynomial of degree $2g$.
\end{conjecture}
An analogous conjecture should hold for the Voros coefficients, but we do not state it here for brevity of notation and to avoid a few subtleties. We have checked that Conjecture \ref{conj:main} holds for low values of $g$ for all curves of hypergeometric type. We hope to return to its proof in a future paper.

\subsection{Other possible $\Omega$}
Let us finally explain to what extent the BPS structure we have proposed is not unique. That is: given the particular form of the free energy we have given, to what extent would other values of $\Omega$ continue to make our formula hold?

\begin{proposition}\label{prop:Blemma}
Suppose $(\Gamma$, $Z$, $\Omega)$ is a refined BPS structure hypergeometric type, and $\zbar$ any quantum correction to the central charge. If 
\begin{equation}\overline{\Omega}:\Gamma\rightarrow \mathbb{Z}[t,t^{-1}]
\end{equation}
vanishes for all but finitely many $\gamma\in\Gamma$ and satisfies  \begin{equation}\label{eq:omegabar}\sum_{\gamma\in\Gamma} \sum_{{ n \in \mathbb{Z}}} \dfrac{{\mathsf{B}_{2,2g}^{[n]}(\gamma)}}{4g(2g-1)(2g-2)}\overline{\Omega}_n(\gamma)\left( 
{\dfrac{2 \pi i}{Z(\gamma)}} \right)^{2g-2}=\sum_{\gamma\in\Gamma} \sum_{{ n \in \mathbb{Z}}} \dfrac{{\mathsf{B}_{2,2g}^{[n]}(\gamma)}}{4g(2g-1)(2g-2)}{\Omega}_n(\gamma)\left( 
{\dfrac{2 \pi i}{Z(\gamma)}} \right)^{2g-2}\end{equation}
for all $g>1$  and all $m\neq 0$ then
\begin{equation}
\overline{\Omega}_n(\gamma)+\overline{\Omega}_{-n}(\gamma)=\Omega_n(\gamma)+\Omega_{-n}(\gamma)
\end{equation}
for all $n\in \mathbb{Z}$.
\end{proposition}
\begin{proof}
Suppose for some $\overline{\Omega}:\Gamma\rightarrow \mathbb{Z}$ and all $g>1$ that \eqref{eq:omegabar} holds, i.e. \begin{align}
   \sum_{\gamma\in\Gamma}\sum_{n\in\mathbb{Z}}c_n(\gamma)\frac{\mathsf{B}_{2,2g}^{[n]}(\gamma)}{4g(2g-1)(2g-2)}\left(\frac{2\pi i}{Z(\gamma)}\right)^{2g-2}=0
    \end{align}
    where $c_n(\gamma)=\overline{\Omega}_n(\gamma)-\Omega_n(\gamma)$.

\noindent Consider first the case that there is only one active class $\gamma_{\pm}$ up to sign (e.g. Weber, Whittaker, Bessel). Then we have that for $g>1$,
\begin{align}
\sum_{\gamma=\gamma_\pm}\left(\sum_{n}c_n(\gamma)\frac{\mathsf{B}_{2,2g}^{[n]}(\gamma)}{4g(2g-1)(2g-2)} \right)\left(\frac{2\pi i}{Z(\gamma)}\right)^{2g-2}=0\\
    \sum_{n}\frac{c_n(\gamma)}{2} \sum_{\gamma=\gamma_\pm}\left(\frac{2\pi i}{Z(\gamma)}\right)^{-2}\frac{\mathsf{B}_{2,2g}^{[n]}(\gamma)}{(2g)!}\left(\frac{2\pi i}{Z(\gamma)}\right)^{2g}=0\\                                         \sum_{n\geq0}\frac{c_n+c_{-n}-c_0\delta_{0,n}}{2}\sum_{\gamma=\gamma_\pm}\left(\frac{2\pi i}{Z(\gamma)}\right)^{-2}\left(\frac{\mathsf{B}_{2,2g}^{[n]}(\gamma)}{(2g)!} \right)\left(\frac{2\pi i}{Z(\gamma)}\right)^{2g}=0,
       \end{align}

  \noindent where we used the double Bernoulli polynomial identity \eqref{eq:refl}, and removed the dependence of $c_n$ on $\gamma$ from the notation since it is independent of the sign. We then sum over all $g\geq0$ and use the definition of the double Bernoulli polynomial \eqref{eq:BNdef} to obtain
\begin{align}
          \sum_{n\geq0}\frac{c_n+c_{-n}-c_0\delta_{0,n}}{2}\sum_{\gamma_\pm}\left(\frac{2\pi i}{Z(\gamma)}\right)^{-2} \sum_{g\geq\frac{3}{2}}\left(\frac{\mathsf{B}_{2,2g}^{[n]}(\gamma)}{(2g)!} \right)\left(\frac{2\pi i}{Z(\gamma)}\right)^{2g}+\left(0,\frac12,1\text{ terms}\right)=0\\
                    \sum_{n\geq0}\frac{(c_n+c_{-n}-c_0\delta_{0,n})}{2}\sum_{\gamma=\gamma_\pm}\frac{ e^{\left((n+1)\frac{\mathscr{Q}}{2}+\frac{\zbar(\gamma)}{2\pi i}\right)\left(\frac{2\pi i}{Z(\gamma)}\right)}}{\left(e^{\sqrt{\beta}\frac{2\pi i}{Z(\gamma)}}-1\right)\left(e^{- \frac{1}{\sqrt{\beta}}\frac{2\pi i}{Z(\gamma)}}-1\right)} +\left(0,\frac12,1\text{ terms}\right) =0\\
                \frac{ e^{\left(\frac{\mathscr{Q}}{2}+\frac{\zbar(\gamma)}{2\pi i}\right)\left(\frac{2\pi i}{Z(\gamma)}\right)}}{\left(e^{\sqrt{\beta}\frac{2\pi i}{Z(\gamma)}}-1\right)\left(e^{- \frac{1}{\sqrt{\beta}}\frac{2\pi i}{Z(\gamma)}}-1\right)} \sum_{n\geq0}\frac{(c_n+c_{-n}-c_0\delta_{0,n})}{2}2\cosh{\frac{n \mathscr{Q} \pi i }{Z(\gamma)}} +\left(0,\frac12,1\text{ terms}\right)= 0.
\end{align}
where in the last line $\gamma$ can be either of $\gamma_+$ or $\gamma_-=-\gamma_+$. The prefactor never vanishes, but the functions (of $\mathscr{Q}$) inside the sum (and the low order terms) are clearly linearly independent, so that all $c_n+c_{-n}-c_0 \delta_{0,n}$ must vanish.

Finally, since we only considered hypergeometric type cases, two active elements $\gamma_1$, $\gamma_2\in\Gamma$ satisfying $Z(\gamma_1)=Z(\gamma_2)$ must be equal, so our argument holds termwise for each (pair of) term in the sum over $\Gamma$.
\end{proof}
{A similar result holds for the Voros coefficient}. Clearly, for $\Omega_0$ there is a unique possible value, but for other values of $n$ there is not (even if we consider only integral BPS structures). We wrote our formula in terms of a particular $\Omega$ which seemed natural to us based on hints in the literature, but it could be written just as well by any of these equivalent $\overline{\Omega}$. One possibility to select our own choice is to require $\Omega$ is palindromic for type I, II, III saddles, i.e. in the absence of ring domains (distinguishing these cases is natural from the BPS-theoretic and QFT point of view). This leaves the value of $\Omega(\gamma)=-t^{-1}$ (for $\gamma$ arising from type IV saddles) unexplained, since it is easy to check no palindromic integer-valued $\Omega$ will agree with Conjecture \ref{conj:main} --- in this case, we chose it in order to agree with the DT invariant obtained in \cite{kidwai:toappear}. We leave a deeper understanding to future work.

\appendix

\section{Multiple Bernoulli polynomials}\label{sec:Bernoulli}
Our results are phrased in terms of multiple Bernoulli polynomials, and we use several identities between such expressions repeatedly throughout the text. Let us recall the definitions, and state some identities.


For any integer $N\geq1$, fix a tuple of nonvanishing complex numbers $(a_1,\ldots,a_N)$. The multiple Bernoulli polynomial is defined by
\begin{equation}\label{eq:BNdef}
    \frac{w^Ne^{xw}}{(e^{a_1 w}-1)\ldots(e^{a_N w}-1)}=\sum_{k\geq0}B_{N,k}(x\,|\,a_1,\ldots,a_N)\cdot \frac{w^k}{k!}
\end{equation}

\noindent Note that $B_{1,k}(x\,|\,1)=B(x)$ is the (usual) Bernoulli polynomial. We refer to the cases $N=1$ and $N=2$ as the single and double Bernoulli polynomials, respectively.

The multiple Bernoulli polynomials satisfy the rescaling property
\begin{equation}\label{eq:rescaling}
    B_{N,k}(\lambda\,|\,\lambda a_1,\ldots \lambda a_N)=\lambda^{k-N}B_{N,k}(x\,|\,a_1,\ldots a_N).
\end{equation}

\noindent for $\lambda\in\mathbb{C}^*$. 

For our choice of parameters $a_1=\beta^\frac12$, $a_2=-\beta^{-\frac12}$ we can write the important symmetry relation
\begin{equation}\label{eq:refl}
    B_{2,k}\left(-x+\frac{\mathscr{Q}}{2}\,\Big|\,\beta^{\frac12},-\beta^{-\frac12}\right)=    (-1)^kB_{2,k}\left(x+\frac{
    \mathscr{Q}}{2}\,\Big|\,\beta^\frac12,-\beta^{-\frac12}\right)
\end{equation}
which is often useful to write as the $\mathscr{Q}$-shift:
{\begin{equation}\label{eq:2Bidentity}
     B_{2,k}\left(-x+\mathscr{Q}\,\big|\,\beta^{\frac12},-\beta^{-\frac12}\right)=(-1)^k B_{2,k}\left(x\,\big|\,\beta^{\frac12},-\beta^{-\frac12}\right).
\end{equation}}
Similarly, for $B_{1,k}$ we have
\begin{equation}
(-1)^kB_{1,k}(-x\,|\,\beta^\frac12)=B_{1,k}(x+\beta^\frac12\,|\,\beta^\frac12)
\end{equation}
which is a rewriting of the basic identity for (ordinary) Bernoulli polynomials.





\section{Proof for Airy curve and degenerate Bessel curve}\label{appendix:airydbes}
We will prove that $F_g=0$ for both Airy curve and degenerate Bessel curve. Since similar arguments can be applied to both cases, we only focus on the Airy curve.

For the Airy curve, each $\omega_{g,n}(p_1,..,p_n)$ can be written as a rational differential in $y_i:=y(p_i)$. This is because the function field of $\mathcal{C}$ is generated by $x,y$ (c.f. Assumption 2.7 in \cite{kidwai2023quantum}) and $x=y^2$ for the Airy curve. In particular this implies that for each $2g,n\in\mathbb{Z}_{\geq0}$ there exists a rational function $W_{g,n+1}(y_0,..,y_n)$ satisfying
\begin{equation}
    \omega_{g,n+1}(p_0,..,p_n)=W_{g,n+1}(y_0,..,y_n)dy_0\cdots dy_n.
\end{equation}
For example, we have
\begin{equation}
     \omega_{0,1}(p_0)=2y_0^2dy_0,\quad \omega_{0,2}(p_0,p_1)=\frac{dy_0dy_1}{(y_0+y_1)^2},\quad \omega_{\frac12,1}(p_0)=-\frac{\mathscr{Q}}{2}\frac{dy_0}{y_0}.\label{B1}
\end{equation}

\begin{lemma}\label{lem:monomial}
       Let $c\in\mathbb{C}^*$. Then for all $2g,n\in\mathbb{Z}_{\geq0}$, we have:
    \begin{equation}
        W_{g,n+1}(cy_0,..,cy_n)=c^{2-6g-4n}W_{g,n+1}(y_0,..,y_n)
    \end{equation}
\end{lemma}
\begin{proof}
It is clear from \eqref{B1} that the statement holds for $(g,n)=(0,0), (0,1), (\frac12,0)$. For others, it is easy to see due to the presence of $\omega_{0,1}$ in the denominator of the recursion formula (Definition \ref{def:RTR}) that if $\omega_{0,1}$ is scaled by $c\in\mathbb{C}^*$, then $\omega_{g,n}$ is accordingly scaled by $c^{2-2g-n}$. Since $y\to c y$ is equivalent to $\omega_{0,1}\to c^3\omega_{0,1}$, this implies the lemma.
\end{proof}

\begin{proposition}\label{prop:dbesairy}
     $F_g=0$ for the Airy curve.
\end{proposition}
\begin{proof}
    Recall from Definition \ref{def:Funstable} that $F_0,F_\frac12,F_1$ for the Weber and Whittaker curve are defined as a solution of differential equations with respect to parameters of the corresponding spectral curve. Since there is no parameter in the Airy curve, we define those to be zero. Thus, we will show that $F_g=0$ for $g>1$.

    Since $\omega_{g,1}$ for $g>1$ has a pole only at the preimage of $y(p)=0$ which is none other than the effective ramification point, Lemma \ref{lem:monomial} implies that there exists a constant $c_{g,1}\in\mathbb{C}[\mathscr{Q}]$ such that
    \begin{equation}
        \omega_{g,1}(p)=c_{g,1}\frac{dy}{y^{6g-2}}.
    \end{equation}
    Now for the Airy curve, we can explicitly write $\Phi(p)$ as
    \begin{equation}
        \Phi(p)=y^3+c_\phi,
    \end{equation}
    where $c_\phi$ is some constant. Therefore, we find
    \begin{equation}
        \Phi(p)\cdot\omega_{g,1}(p)=c_{g,1}\cdot \left(\frac{1}{y^{6g-5}}+\frac{c_\phi}{y^{6g-2}}\right)dy.
    \end{equation}
    For $g>1$, this clearly indicates that $\Phi(p)\cdot\omega_{g,1}(p)$ has no residue, hence we have $F_g=0$.
\end{proof}
\begin{remark}
    For the degenerate Bessel curve, a similar argument can be applied by replacing $y$ with $y^{-1}$. The degree of homogeneity of $W_{g,n+1}$. becomes $2g-2$ instead of $2-6g-4n$, and the effective ramification point will become the pole of $y(p)$ instead of the zero of $y(p)$. Then, we can show that $\Phi(p)\cdot\omega_{g,1}(p)$ has no residue by simply counting the pole degree and we obtain that $F_g=0$ for the degenerate Bessel curve too.
\end{remark}

\bibliography{refs}

\begin{thebibliography}{10}
\bibitem{koikeschafke}
T.~Koike and R.~Sch\"afke, ``{On the Borel summability of WKB solutions of
  Sch\"odinger equations with polynomial potentials and its application}.'' in
  preparation.
\end{thebibliography}
\bibliographystyle{utphys.bst}

\end{document}